\documentclass[11pt,a4paper]{amsart}
\usepackage[T1]{fontenc}

\usepackage{mathtools}
\usepackage{pict2e}
\usepackage{enumitem}
\usepackage{latexsym}
\usepackage{mathrsfs}
\usepackage{textcomp}
\usepackage[table]{xcolor}
\usepackage{pifont}
\usepackage{multirow}
\usepackage{url}

\usepackage[all]{xy}

\usepackage{lmodern}
\usepackage{graphicx}
\usepackage{booktabs}
\usepackage{calc}
\usepackage{adjustbox}

\usepackage{tikz-cd}

\usepackage[hidelinks]{hyperref}

\setlist[enumerate, 1]{label=(\roman*), wide, itemsep=0pt, topsep=0pt}
\setlist[enumerate, 2]{label=(\arabic*), wide=2\parindent, itemsep=0pt, topsep=0pt}
\tikzset{myloop/.style={in=-30, out=30, looseness=7}} 
\tikzset{myloop2/.style={in=-150, out=150, looseness=7}} 

\theoremstyle{plain}
\newtheorem{theorem}{Theorem}[section]
\theoremstyle{plain}

\newtheorem{proposition}[theorem]{Proposition}
\newtheorem{lemma}[theorem]{Lemma}

\newtheorem{corollary}[theorem]{Corollary}

\theoremstyle{remark}
\newtheorem{remark}[theorem]{Remark}
\newtheorem{example}[theorem]{Example}
\theoremstyle{definition}
\newtheorem{definition}[theorem]{Definition}

\numberwithin{equation}{section}

\DeclarePairedDelimiter\gen{\langle}{\rangle}
\DeclarePairedDelimiter\Set{\{}{\}}

\DeclarePairedDelimiter\comp{\langle\!\langle}{\rangle\!\rangle}

\DeclareMathOperator{\characteristic}{char}

\DeclareMathOperator{\proj}{proj}

\DeclareMathOperator{\md}{mod}
\DeclareMathOperator{\Mod}{Mod}

\DeclareMathOperator{\rad}{rad}

\DeclareMathOperator{\g}{{\rm\bf g}}
\DeclareMathOperator{\Jacobi}{\mathcal{P}}
\DeclareMathOperator{\im}{im}

\DeclareMathOperator{\Ext}{Ext}

\DeclareMathOperator{\coker}{coker}

\let\hom\relax
\DeclareMathOperator{\hom}{Hom}
\DeclareMathOperator{\End}{End}

\DeclareMathOperator{\pr}{pr}
\DeclareMathOperator{\Dim}{\underline{dim}}

\DeclareMathOperator{\op}{op}

\DeclareMathOperator{\Dis}{Dis}

\newcommand{\Z}{\mathbb Z}

\newcommand{\Gam}{\Gamma}
\newcommand{\Lam}{\Lambda}
\newcommand{\bLam}{\overline{\Lam}}
\newcommand{\ebrace}[1]{\langle #1\rangle}
\newcommand{\fJ}{\mathfrak{J}}
\newcommand{\rk}{\operatorname{rank}}  

\begin{document}

\title[Minimal presentations for representations of pseudocompact algebras]{Minimal presentations of finite-dimensional representations of pseudocompact algebras}
\date{October 13, 2026}



\author{Christof Gei{\ss}}
\address{Christof Gei{\ss}\newline
Instituto de Matem\'aticas\newline
Universidad Nacional Aut{\'o}noma de M{\'e}xico\newline
Ciudad Universitaria\newline
04510 M{\'e}xico D.F.\newline
M{\'e}xico}
\email{christof.geiss@im.unam.mx}

\author{Carolin Hartung}
\address{Carolin Hartung\newline
Mathematisches Institut\newline
Universit\"at Bonn\newline
Endenicher Allee 60\newline
53115 Bonn\newline
Germany}
\email{hartung@math.uni-bonn.de}

\author{Daniel Labardini-Fragoso}
\address{Daniel Labardini-Fragoso\newline
Dipartimento di Matematica 'Tullio Levi-Civita'\newline
Università degli Studi di Padova\newline
Via Trieste, 63 - 35121 Padova\newline
Italy
}
\email{daniel.labardinifragoso@unipd.it}

\author{Jan Schr\"oer}
\address{Jan Schr\"oer\newline
Mathematisches Institut\newline
Universit\"at Bonn\newline
Endenicher Allee 60\newline
53115 Bonn\newline
Germany}
\email{schroer@math.uni-bonn.de}

\begin{abstract}
    We give a general formula for minimal projective presentations and for minimal discrete-injective presentations
    of finite-dimensional modules over basic pseudocompact algebras. We also show that the representation theoretic $g$-vector 
    of a module coincides with its DWZ-$g$-vector. These results generalize the description that Derksen, Weyman and Zelevinsky gave of minimal projective presentations for modules over certain finite-dimensional Jacobian algebras.

    In particular, we can drop their finiteness assumption and we may work over fields of arbitrary characteristic.
    
\end{abstract}

\maketitle
\tableofcontents

\section{Introduction and main results}

\subsection{Notation} \label{ssec:notn}
We fix the following notation: Throughout, $K$ denotes a field. 
Let $Q=(Q_0, Q_1, s,t)$ be a finite quiver, and let
$\Gam \coloneq K\comp{Q}$ be the completed path algebra of $Q$ with Jacobson radical 
$\mathfrak{m}$. Note that the ideal $\mathfrak{m}$ is generated by the set of arrows. 
For an ideal $\mathfrak{J}\coloneq\ebrace{R} \subseteq \mathfrak{m}^2\subset \Gam$, the algebra 
$\Lam \coloneq \Gam/\mathfrak{J}$ is basic and possibly of infinite dimension. Moreover, let $\bLam \coloneq \Gam/\overline{\mathfrak{J}}$, where $\overline{\mathfrak{J}}$ is the closure of 
$\mathfrak{J}$ with respect to the $\mathfrak{m}$-adic topology. Then 
$\bLam$ is a basic pseudocompact algebra in the sense of Gabriel \cite{GabrielP.1962Dca}. See Section~2 of~\cite{SimsonDaniel2001Ccpa} for
a gentle introduction to pseudocompact algebras and their categories of
pseudocompact resp. discrete modules.

For a $K$-algebra $B$ let ${\rm mod}(B)$ be the category of finite-dimensional $B$-modules.
The category of all $B$-modules is denoted by $\Mod(B)$.

By a \emph{module} we always mean a finite-dimensional left module, unless stated otherwise.

For the algebras defined above we have ${\rm mod}(\Lam) = {\rm mod}(\bLam)$.
On the other hand, the categories $\Mod(\Lam)$ and $\Mod(\bLam)$ do not coincide in general.

Note that the indecomposable projective or injective $\Lam$-modules or
$\bLam$-modules are possibly of infinite dimension. 
We write quiver representations of the form $M=(M(i),M(\alpha))_{i \in Q_0, \alpha \in Q_1}$. Moreover, we identify representations of quivers with modules over the corresponding path algebra.
The finite-dimensional nilpotent representations of quivers correspond then to finite-dimensional modules over the completed
path algebra.

For tensor products over our field $K$ we often write $\otimes$ instead of $\otimes_K$.

\subsection{Representation theoretic $g$-vectors}
With 
$\bLam = \Gam/\overline{\mathfrak{J}}$ as in Section~\ref{ssec:notn}.
Let $n = |Q_0|$ be the number of vertices of $Q$.
The algebra $\bLam$ has then $n$ indecomposable projective modules, up to isomorphism,
namely
$P_i \coloneq \bLam e_i$ with $i \in Q_0$ and $e_i$ the corresponding idempotent path of length $0$.
The module $P_i$ is the projective cover of the simple $\bLam$-module $S_i$ associated to $i$.
Let
${\rm proj}(\bLam)$ be the category of finitely generated projective $\bLam$-modules.
The Grothendieck group $K_0(\proj(\bLam))$ can be identified with $\Z^n$ and has the classes
$[P_i]$ with $i \in Q_0$ as a $\Z$-basis.

For $i \in Q_0$ let $I_{i,{\rm dis}} \coloneq D_2(e_i\bLam)$ be the corresponding discrete
indecomposable injective module. (The duality $D_2$ is defined in \cite{SimsonDaniel2001Ccpa}.)

Each module in ${\rm mod}(\bLam)$ is finitely presented.
For $M \in {\rm mod}(\bLam)$ let
\[
\bigoplus_{i \in Q_0} P_i^{a_i} \to \bigoplus_{i \in Q_0} P_i^{b_i} \to M \to 0
\]
be a minimal projective presentation.
Then the \emph{representation theoretic $g$-vector} $\g^{\rm proj}(M)$ is defined as
$$
\g^{\rm proj}(M)\coloneq \g_{\bLam}^{\rm proj}(M) \coloneq (a_i-b_i)_{i \in Q_0}.
$$

Dually, for $M$ in ${\rm mod}(\bLam)$
let
\[
0 \to M \to \bigoplus_{i \in Q_0} I_{i,{\rm dis}}^{c_i} \to \bigoplus_{i \in Q_0} 
I_{i,{\rm dis}}^{d_i} 
\]
be a minimal injective presentation in the category ${\rm Dic}(\bLam)$
of discrete $\bLam$-modules. 
Note that this might differ from the minimal
injective presentation in $\Mod(\bLam)$.
Then let 
$$
\g^{\rm inj}(M) \coloneq \g_{\bLam}^{\rm inj}(M) \coloneq (d_i-c_i)_{i \in Q_0}.
$$

\subsection{Jacobian algebras and projective presentations} To understand the main result of this paper, it is helpful if we first outline how to construct minimal projective presentations (resp. minimal injective presentations) for finite-dimensional modules over Jacobian algebras. 

Let $B \coloneq \bLam\coloneq \Jacobi_K(Q,W)$ be a Jacobian algebra in the sense of 
\cite[Def. 3.1]{DWZ08}. Thus,  $B$ is in particular a basic pseudocompact algebra in the sense. The construction of minimal projective presentations (resp. minimal injective presentations) over certain finite-dimensional Jacobian algebras 
goes back to the celebrated paper~\cite[Sec.~10]{DWZ08}. In this paper Derksen, Weyman, and Zelevinsky introduced for each 
$M \in {\rm mod}(B)$ and for each vertex $k \in Q_0$ an elementary defined complex of vector spaces
\begin{equation}
    \label{eq:complex:jacobian}
    0 \rightarrow M(k) \xrightarrow{M(\beta_k)} M_{\text{out}}(k) \xrightarrow{M(\gamma_k)} M_{\text{in}}(k) \xrightarrow{M(\alpha_k)} M(k) \rightarrow 0.
\end{equation}
This depends functorially on $M$. It is not necessarily exact. Let us set
\begin{align*}
    h_k^{(0)}(M) &\coloneq \dim\ker(M(\beta_k)), \\
    h_k^{(1)}(M) &\coloneq \dim\ker(M(\gamma_k))/\im(M(\beta_k)), \\
    h_k^{(2)}(M) &\coloneq \dim\ker(M(\alpha_k))/\im(M(\gamma_k)), \\
    h_k^{(3)}(M) &\coloneq \dim\coker(M(\alpha_k)).
\end{align*}
By setting
\begin{align*}
    \g_k^{\rm in}(M) &\coloneq h_k^{(2)}(M) - h_k^{(3)}(M) = \dim\coker(M(\gamma_k))-\dim M(k), \\
    \g_k^{\rm out}(M) &\coloneq h_k^{(1)}(M) - h_k^{(0)}(M) = \dim\ker(M(\gamma_k))-\dim M(k),
\end{align*}
we obtain the \emph{$g$-vector} $$\g^{\rm in}(M) \coloneq \g_{\bLam}^{\rm in}(M) \coloneq (\g_k^{\rm in}(M))_{k \in Q_0}.$$ 
Note that in \cite{DWZ10} the $g$-vector is a key device for the additive categorification of skew-symmetric cluster algebras.
We also call $\g^{\rm in}(M)$ the \emph{DWZ-$g$-vector} of $M$.
Let $$\g^{\rm out}(M) \coloneq \g_{\bLam}^{\rm out}(M) \coloneq (\g_k^{\rm out}(M))_{k \in Q_0}.$$

Warning: In fact, the g-vector $\g(M)$ of~\cite{DWZ08} is our
vector $\g^{\rm out}(M)$.

For certain finite-dimensional Jacobian algebras $B$,
Derksen--Weyman--Zelevinsky \cite{DWZ10} show that
the DWZ-$g$-vector $\g^{\rm in}(M)$ coincides with the representation theoretic $g$-vector $\g^{\rm proj}(M)$ 
for each $M \in {\rm mod}(B)$.

Namely, they construct from the above complex in an elementary way a minimal projective presentation of the form
\begin{equation}
    \label{eq:min:proj:pres:jacobian}
    \bigoplus_{k \in Q_0} (B e_k)^{h_k^{(2)}(M)} \rightarrow \bigoplus_{k \in Q_0} (B e_k)^{h_k^{(3)}(M)} \rightarrow M \rightarrow 0.
\end{equation}
By duality, the authors also obtain a minimal injective presentation
\begin{equation}
    \label{eq:min:inj:copres:jacobian}
    0 \rightarrow M \rightarrow \bigoplus_{k \in Q_0} D(e_k B)^{h_k^{(0)}(M)} \rightarrow \bigoplus_{k \in Q_0} D(e_k B)^{h_k^{(1)}(M)}.
\end{equation}
In particular, as mentioned above, it follows for those Jacobian algebras that
\begin{equation*}
    \g^{\rm proj}(M) = \g^{\rm in}(M) \text{ and } \g^{\rm inj}(M) = \g^{\rm out}(M).
\end{equation*}

\subsection{Main result}
One of our main aims is to define DWZ-$g$-vectors $\g^{\rm in}(M)$ for finite-dimensional modules $M$
over arbitrary basic pseudocompact algebras $\bLam$ as defined in Section~\ref{ssec:notn} and to show that these coincide with the 
representation theoretic $g$-vectors $\g^{\rm proj}(M)$.
In particular, the finite-dimension hypothesis made in \cite[eq. 10.3]{DWZ10} can be dropped.

 In order to generalize (\ref{eq:complex:jacobian}) to the setting of Section~\ref{ssec:notn}, we construct two complexes in terms of generators of $\mathfrak{J}$, namely
\begin{equation*}
    M'_{\text{rel}}(k) \xrightarrow{M(\gamma'_k)} M_{\text{in}}(k) \xrightarrow{M(\alpha_k)} M(k) \rightarrow 0,
\end{equation*}
and
\begin{equation*}
    0 \rightarrow M(k) \xrightarrow{M(\beta_k)} M_{\text{out}}(k) \xrightarrow{M(\gamma''_k)} M''_{\text{rel}}(k). 
\end{equation*}
In the case of $\bLam = \Jacobi_K(Q,W)$ these specialize to the right and left side of the complex in (\ref{eq:complex:jacobian}). Accordingly we define now the numbers $h_k^{(i)}(M)$ as above, with the obvious modifications
\begin{align*}
    h_k^{(1)}(M) &\coloneq \dim\ker(M(\gamma''_k))/\im(M(\beta_k)) \text{ and } \\
    h_k^{(2)}(M) &\coloneq \dim\ker(M(\alpha_k))/\im(M(\gamma'_k)).
\end{align*}
Also the DWZ-$g$-vectors $\g^{\rm in}(M)$ and $\g^{\rm out}(M)$ are defined as before by
\begin{align*}
    \g_k^{\rm in}(M) &\coloneq h_k^{(2)}(M) - h_k^{(3)}(M), \\
    \g_k^{\rm out}(M) &\coloneq h_k^{(1)}(M) - h_k^{(0)}(M).
\end{align*}
This leads to our main result, which generalizes the construction of the minimal projective presentation in (\ref{eq:min:proj:pres:jacobian}):

\begin{theorem}\label{thm:main1}
    Let $\mathfrak{J} \subseteq \mathfrak{m}^2$ be an ideal of $\Gam = K\comp{Q}$, and let $M$ be a finite-dimensional representation of $\Lam = \Gam/\mathfrak{J}$. Then
    $M$ has a minimal projective presentation of the form
    \begin{equation}\label{eq:main1a}
        \bigoplus_{k \in Q_0} (\Lam e_k)^{h_k^{(2)}(M)} \rightarrow \bigoplus_{k \in Q_0} (\Lam e_k)^{h_k^{(3)}(M)} \rightarrow M \rightarrow 0
    \end{equation}
    in $\Mod(\Lam)$.
In particular, we have \[\g^{\rm proj}(M) = \g^{\rm in}(M).\]    
Let $\bLam\coloneq \Gam/\overline{\mathfrak{J}}$, and let
\[ D_2\colon {\rm PC}(\bLam^{\rm op}) \to {\rm Dis}(\bLam) \]
denotes the (continuous) duality between the category of pseudocompact $\bLam^{\rm op}$-modules and the category of 
discrete $\bLam$-modules.
Then the finite-dimensional representation $M$ descends to $\bLam$, and we find a 
minimal projective presentation 
\begin{equation}\label{eq:main1b}
        \bigoplus_{k \in Q_0} (\bLam e_k)^{h_k^{(2)}(M)} \rightarrow \bigoplus_{k \in Q_0} (\bLam e_k)^{h_k^{(3)}(M)} \rightarrow M \rightarrow 0
    \end{equation} 
    in ${\rm PC}(\bLam)$ (and $\Mod(\bLam)$)
and also a minimal
injective presentation
\begin{equation}\label{eq:main1c}
0\rightarrow M\rightarrow \bigoplus_{k\in Q_0} D_2(e_k \bLam)^{h^{(0)}_k(M)}
\rightarrow  \bigoplus_{k\in Q_0} D_2(e_k \bLam)^{h^{(1)}_k(M)}
\end{equation}
in ${\rm Dis}(\bLam)$.
In particular, we have 
$$
\g^{\rm inj}(M) = \g^{\rm out}(M).
$$
\end{theorem}

The proof follows closely the ideas of the proof of \cite[Prop. 10.4]{DWZ10}, however we use the (non-commutative) Nakayama Lemma instead of the nilpotency of the Jacobson radical assumed in \cite[eq. 10.3]{DWZ10}. Moreover, we need to extract from \cite{BK99} a suitable version of the beginning of a standard bimodule resolution, namely for possibly infinite-dimensional quotients of completed path algebras.

The minimal injective presentation (\ref{eq:main1c})
is a generalized analogue of the minimal injective presentation constructed in
    \cite[Rem. 10.8]{DWZ10}.

\begin{remark} 
This last result clarifies and sharpens \cite[Lemma 3.1]{CLS15}. Note that our proof here is quite direct and elementary, whilst the proof in \cite{CLS15} is based on difficult results of Plamondon \cite[Lemma 4.7, Proposition 4.8]{PlamondonPierre-Guy2011Cavc} on cluster categories with infinite-dimensional Hom-spaces.
\end{remark}


Each finite-dimensional representation $M$ of $\Lam$ is \emph{nilpotent},
\emph{i.e.} there exists a $l_0$, depending on $M$, such that 
$\mathfrak{m}^l M = 0$ for all $l\geq l_0$.  
Consider now the finite-dimensional algebra 
\[
\Lam_l \coloneq \Gam/(\mathfrak{J}+\mathfrak{m}^l).
\]
We call this a \emph{truncation} of $\Lam$.
Then $M$ obviously
descends to a representation of $\Lam_l$ for $l\geq l_0$, and
we can apply Theorem~\ref{thm:main1} with $\Lam_l$ in place of $\Lam$.
The precise statement of our result in Section~\ref{sec:main} shows
that, for $l> l_0$, the invariants $h_k^{(i)}(M)$ do not depend
on $l$.  In particular the DWZ-$g$-vectors $\g^{\rm in}(M)$ and $\g^{\rm out}(M)$ can be calculated
already with any truncation parameter $l > l_0$,  see Remark~\ref{remark:dim:min:proj} for more details. 

The truncation technique for studying $g$-vectors of finite-dimensional representations 
is not new and can already be found in \cite{CLS15}.

\section{Completed path algebras and the Nakayama Lemma}

\subsection{Completed path algebras and relations}
\label{section:completed:path:algebras}
We fix the following notation: Let $K$ be a field, and let $Q=(Q_0,Q_1,s,t)$ be a finite quiver with vertex set $Q_0 = \Set{1, \dots, n}$ and arrow set $Q_1$. The functions $s,t: Q_1 \rightarrow Q_0$ determine the start and terminal vertex of each arrow, i.e. an arrow $\alpha \in Q_1$ starts in $s(\alpha)$ and ends in $t(\alpha)$. Let $e_1, \dots, e_n$ be the paths of length $0$ in $Q$. They form a complete set of pairwise orthogonal primitive idempotents in $A$. 

The path algebra of $Q$ is denoted by $KQ$. As before, let $\Gam = K\comp{Q}$ be the \emph{completed path algebra} of $Q$. As a $K$-vector space we have
\begin{equation*}
    \Gam = \prod_{m \geq 0}KQ_m
\end{equation*}
where $KQ_m$ is the $K$-vector space with a basis labeled by the paths of length $m$ in $Q$. The multiplication of $KQ$ and $\Gam$ is induced by concatenation of paths.
Thus, the elements of $\Gam$ are (possibly infinite) $K$-linear combinations of paths in $Q$. Moreover, both algebras are naturally graded by the length of paths.

We view $\Gam$ as a topological space with respect to the $\mathfrak{m}$-adic topology, where the (two-sided) ideal $\mathfrak{m} \subset \Gam$ is given by
\begin{equation}
    \label{eq:ideal:m}
    \mathfrak{m} \coloneq \prod_{m \geq 1} KQ_m,
\end{equation}
called the \emph{arrow ideal} of $\Gam$. To be precise, for any subset $U \subseteq \Gam$ let 
\begin{equation*}
    \overline{U} \coloneq  \bigcap_{p \geq 0} (U + \mathfrak{m}^p)
\end{equation*}
be the \emph{$\mathfrak{m}$-adic closure} of $U$.

Let us give another description of path algebras and completed path algebras by identifying them with certain tensor algebras. 
For this we set
\begin{equation*}
    S \coloneq \prod_{i \in Q_0} Ke_i \text{\quad and \quad} A \coloneq \bigoplus_{\alpha \in Q_1}K\alpha.
\end{equation*}
By definition, $S$ is a commutative semisimple $K$-algebra. We can define a $S$-$S$-bimodule structure on $A$ by setting $e_i\alpha e_j=\delta_{i,t(\alpha)}\alpha\delta_{s(\alpha),j}$. 
For $l \ge 0$ let
\[
A^{\otimes_S^l} \coloneq A \otimes_S \cdots \otimes_S A
\]
be the tensor product of $l$ copies of $A$.
We can then identify the tensor algebra
\begin{equation*}
    \Gam^0 \coloneq T_S(A) = \bigoplus_{l\geq 0} A^{\otimes_S^l}
\end{equation*}
with the path algebra $KQ$. The completed tensor algebra
\begin{equation*}
    \Gam \coloneq \hat{T}_S(A) = \prod_{l\geq 0} A^{\otimes_S^l}
\end{equation*}
is identified with the completed path algebra $\Gam$. The arrow ideal gets identified with 
\begin{equation*}
    \mathfrak{m} \coloneq \prod_{l\geq 1} A^{\otimes_S^l} \subset \Gam
\end{equation*}
which is also the Jacobson radical of $\Gam$. As mentioned above, when $\Gam$ is equipped with the $\mathfrak{m}$-adic topology, $\Gam$ is a basic pseudocompact algebra, see 
\cite[Section 3]{GabrielP.1962Dca} and \cite[Section 7]{GabrielP.1973IrII}. Note that 
$\Gam = S \oplus \mathfrak{m}$ as a $S$-$S$-bimodule.

\begin{example}
    The standard example for completed path algebras is the following: Consider the quiver
    \begin{equation*}
        \begin{tikzcd}[sep=scriptsize]
            1 \ar["\alpha", loop right, myloop]
        \end{tikzcd}.        
    \end{equation*}
    We have $Q_0=\Set{1}$ and $Q_1=\Set{\alpha}$. Therefore, we obtain 
    $KQ=K[\alpha]$ and $\Gam = K\comp{Q} = K[[\alpha]]$, the algebra of formal power series.
\end{example}

Next consider a set of \emph{relations} $R \subseteq \mathfrak{m}^2$. We may assume that $\rho=e_{t(\rho)}\rho e_{s(\rho)}$ for all $\rho \in R$. Consider the ideal $\mathfrak{J}=\gen{\rho \mid \rho \in R}$ of $\Gam$ which is generated by the elements of $R$. This is trivially contained in $\mathfrak{m}^2$. 
Moreover, for every $l\geq 0$ we consider the ideals 
\[
\mathfrak{J}_l \coloneq \mathfrak{J} + \mathfrak{m}^l.
\]
The $\mathfrak{m}$-adic closure is then defined to be 
\[
\overline{\mathfrak{J}} \coloneq \bigcap_{l\geq 2} \mathfrak{J}_l.
\]
The algebra $\bLam \coloneq \Gam/\overline{\mathfrak{J}}$ is then also pseudocompact and basic. Moreover, the truncations
\[
\Lam_l \coloneq \Gam/\mathfrak{J}_l
\] 
of $\Lam$
are finite-dimensional and the Jacobson radical $J$ of $\Lam \coloneq \Gam/\mathfrak{J}$ is given by $\mathfrak{m}/\mathfrak{J}$.

\begin{example}
    Jacobian algebras $\Jacobi_K(Q,W)$ are of the form $\bLam = \Gam/\overline{I(W)}$ where $I(W)$ is the ideal generated by the relations coming from the potential $W$. For more details, we refer to Section~\ref{section:jacobian:algebras}.
\end{example}

\subsection{Projective modules}
\label{section:projective:modules}
For each $i \in Q_0$ the $\Gam$-module $P_i = \Gam e_i$ is projective and has endomorphism ring $\End_\Gam(P_i) \cong (e_i\Gam e_i)^{\op}$. Note that $\rad(e_i\Gam e_i)= e_i \mathfrak{m}e_i$ is a maximal ideal, and thus $\End_\Gam(P_i)$ is local. 
It follows that $\Gam$ is a semiperfect algebra. In particular,
each finitely generated projective $\Gam$-module $P$ is isomorphic to $\bigoplus_{i \in Q_0} P_i^{a_i}$ for a unique tuple $(a_i)_{i \in Q_0} \in \mathbb{N}_0^{Q_0}$. Of course, analogous statements hold for $\Lam = \Gam/I$ as above.

\subsection{Nakayama Lemma} Having introduced the necessary notation, we next state the Nakayama Lemma for the convenience of the reader. As mentioned above, in the proof of our main result we are largely inspired by the ideas of \cite[Section 10]{DWZ10}. For our, possibly infinite-dimensional algebras, however, we replace the nilpotency of the Jacobson radical assumed in \cite[eq. 10.3]{DWZ10} by arguments based on a form of the (non-commutative) Nakayama Lemma. 

\begin{lemma}[Nakayama Lemma]
    \label{Nakayama:lemma}
    Let $B$ be a unital ring with Jacobson radical $\mathfrak{m}$, and let $N$ be a finitely generated $B$-module.
    \begin{enumerate}
        \item If $N'$ is a submodule of $\mathfrak{m}N$, then for any submodule $N''$ of $N$ the equality $N'+N''=N$ implies $N''=N$.
        \item If $f \colon M \rightarrow N$ is a module homomorphism such that the induced morphism $\overline{f} \colon M/\mathfrak{m}M \rightarrow N/\mathfrak{m}N$ is surjective, then $f$ was already surjective.
    \end{enumerate}
\end{lemma}

Part (i) is for example a special case of Nakayama's Lemma for modules presented in 
\cite[Section~4.2]{PierceAssocAlg}. Part (ii) is the (easy) Excercise $2$ in \cite{PierceAssocAlg}. For a discussion of the history and attribution of names see for example \cite[Section~4.22]{LamNonComm}.

\section{Discrete and locally nilpotent modules}

In this section, let $\Lam = \Gam/\mathfrak{J}$ and
$\bLam = \Gam/\overline{\mathfrak{J}}$ as before.

\subsection{Characterization of discrete modules}
A module $M \in \Mod(\bLam)$ is \emph{discrete} if for each $v \in M$
the annihilator 
\[
{\rm Ann}_{\bLam}(v) \coloneq \{ a \in \bLam \mid av = 0 \}
\]
is open in $\bLam$.

A module $M \in \Mod(\Gam)$ is
\emph{nilpotent} if there exists some $l_M \ge 0$ such that $\mathfrak{m}^{l_M} M = 0$.
Furthermore, $M$ is \emph{locally nilpotent} if for each
$v \in M$ there exists some $l_v \ge 0$ such that $\mathfrak{m}^{l_v}v = 0$. Clearly, nilpotent modules are locally nilpotent. The converse is in general wrong. By \cite[\S10]{DWZ10}, every finite-dimensional module over a complete path algebra is nilpotent.

\begin{proposition}\label{prop:locnilpotent}
Let $\Gam = K\Gam = K\comp{Q}$, and let $\bLam = \Gam/\overline{\mathfrak{J}}$.
For $M \in \Mod(\Gam)$ the following are equivalent:
\begin{enumerate}

\item\label{item:M-is-loc-nil} $M$ is locally nilpotent;

\item\label{item:long-paths-kill-a-given-v} For each $v \in M$ there exists some $l_v \ge 0$ such that $p v = 0$ for all
paths of length $l_v$ in $Q$;

\item\label{item:every-cyclic-mod-is-f.d.} For each $v \in M$, the module $\Gam v$ is finite-dimensional;

\item\label{item:M-is-union-of-f.d.-mods} $M$ is the union of its finite-dimensional submodules.
    
\end{enumerate}
Furthermore, for $M \in \Mod(\bLam)$ the following are equivalent:
\begin{enumerate}

\item[(a)] $M$ is discrete;

\item[(b)] $M$ is locally nilpotent.
    
\end{enumerate}
\end{proposition}

\begin{proof}
For the first set of conditions, the implications \ref{item:M-is-loc-nil} $\Longrightarrow$ \ref{item:long-paths-kill-a-given-v}, \ref{item:every-cyclic-mod-is-f.d.} $\Longrightarrow$ \ref{item:M-is-union-of-f.d.-mods} and \ref{item:M-is-union-of-f.d.-mods} $\Longrightarrow$ \ref{item:M-is-loc-nil} are clear.

The implication \ref{item:long-paths-kill-a-given-v} $\Longrightarrow$ \ref{item:every-cyclic-mod-is-f.d.} follows from the fact that every element of $\mathfrak{m}^l$ can be written as a finite sum $\sum_{p\in \Pi_l}u_pp$ with $u_p\in \Gam$, which is an easy consequence of the fact that for any given $l>0$ the set $\Pi_l$ of paths on $Q$ of length exactly $l$ is finite (because $Q$ has only finitely many arrows).

To prove that (a) $\Longleftrightarrow$ (b), take a $\bLam$-module $M$ and notice that, since $\bLam$ has been given the quotient topology under the projection $\Gam\rightarrow\bLam$, for any given $m \in M$ 
the annihilator 
${\rm Ann}_{\bLam}(m) $
is open in $\bLam$ if and only if its preimage under said projection contains a positive power $\mathfrak{m}^l$, if and only if the $\Gam$-module homomorphism $\Gam\rightarrow M$ given by $u\mapsto um$ has finite-dimensional image. The equivalence between (a) and (b) then follows from the equivalence of the first set of statements.
\end{proof}



\subsection{Injective versus discrete-injective modules}\label{subsec:injectives}
The injective objects in the category ${\rm Dis}(\bLam)$ are
called \emph{discrete-injective}. Here we give an explicit description of the indecomposable discrete-injective $\bLam$-modules. 
%
%
As before,
$\Lambda_l\coloneq\Gamma/(\mathfrak{J}+\mathfrak{m}^l)$ is the $l^{\operatorname{th}}$ truncation of $\bLam$. 

For each vertex $i\in Q_0$, set
\[
D_2(e_i\bLam) \coloneq \{\varphi\in \hom_{K}(e_i \bLam,K)\mid \varphi \ \text{is continuous}\},
\]
where $\bLam$ is given the quotient topology under the projection $\pi:\Gamma\rightarrow \bLam$ and $K$ is given the discrete topology. Notice that $\varphi\in \hom_{K}(e_i \bLam,K)$ is continuous if and only if $\ker\varphi$ is an open right $\bLam$-submodule of $e_i\bLam$, if and only if $\pi^{-1}(\ker\varphi)$ contains $e_i(\mathfrak{m}^l)$ for $l\gg0$, cf. \cite[\S2]{SimsonDaniel2001Ccpa}. 

Now, for each positive integer $l$, let 
\[
I_{i,l} \coloneq \{\varphi\in\hom_{K}(e_i\Gamma,K) \mid \varphi(e_i(\mathfrak{J}+\mathfrak{m}^l))=0\}.
\]
Then $I_{i,l}$ is an indecomposable injective $\Lambda_l$-module and there is an inclusion of vector spaces 
$I_{i,l} \hookrightarrow I_{i,l+1}$ that is actually an injective envelope in $\md(\Lambda_{l+1})$. Furthermore, $S_i \coloneq I_{i,l}$ is the one-dimensional simple $\bLam$-module supported at $i$ and annihilated by $\mathfrak{m}$. 
Let also $S_i \hookrightarrow I_{i,\bLam}$ be an injective envelope of $S_i$ in $\Mod(\bLam)$. Then we can recursively define $\bLam$-module homomorphisms $\iota_l: I_{i,l}\rightarrow I_{i,\bLam}$ so as to obtain
a commutative diagram 
$$
\begin{tikzcd}        
	0 \arrow{r} &  S_i \arrow[hook]{r} \arrow[hook]{drrr} & I_{i,2}\arrow[hook]{r} \arrow[hook]{drr}{\iota_2} & \cdots \arrow[hook]{r} & I_{i,l} \arrow[hook]{r} \arrow[hook]{d}{\iota_l} & I_{i,l+1} \arrow[hook]{r} \arrow[hook]{dl}{\iota_{l+1}} & \cdots\\
	& & & & I_{i,\bLam} & 
\end{tikzcd}
$$
all of whose horizontal arrows are inclusions and whose non-horizontal arrows are injective envelopes in $\Mod(\bLam)$. The direct limit in $\Mod(\bLam)$ of the horizontal inclusions is
\begin{align*}
\varinjlim_l I_{i,l} & =\bigcup_{l\geq 1} I_{i,l} \\
&= \{\varphi\in \hom_{K}(e_i \Gamma,K)\mid \exists l\gg0 \ \text{such that} \ \varphi(e_i(\mathfrak{J}+\mathfrak{m}^l)) = 0\} \\
&= D_2(e_i \Lambda).
\end{align*}
Notice that $\bigcup_{l\geq 1} I_{i,l}$ is locally nilpotent.

The non-horizontal maps $\iota_l$ give rise to an injective envelope in $\Mod(\bLam)$
$$\iota:\varinjlim_l I_{i,l}\hookrightarrow I_{i,\bLam},$$ 
which means that $D_2(e_i\Lambda)$ is contained in the locally nilpotent part of $I_{i,\bLam}$. In order to deduce that $D_2(e_i\Lambda)$ is the whole locally nilpotent part of $I_{i,\bLam}$, we need the following lemma.

%

\begin{lemma}\label{lemma:how-small-inj-sits-inside-bigger-inj}
	Take $l\geq 1$ and an injective left $\Lambda_l$-module $I$. If $\iota:I\hookrightarrow E$ is an injective envelope of $I$ in $\Mod(\bLam)$, then the image of $\iota$ is
	$$
	\iota(I)=\{m\in E\mid \mathfrak{m}^{l}m=0\}.
	$$
\end{lemma}

\begin{proof}
	Write $E_{(l)} \coloneq \{m\in E\mid \mathfrak{m}^{l}m=0\}$. It is clear that $E_{(l)}$ is a $\Lambda_l$-module and $\iota(I)$ is a $\Lambda_l$-submodule of $E_{(l)}$.  Now, $E_{(l)}$ is a $\bLam$-submodule of $E$, hence 
	$$\operatorname{soc}_{\Lambda_{l}}(\iota(I))\subseteq\operatorname{soc}_{\Lambda_{l}}(E_{(l)})=\operatorname{soc}_{\bLam}(E_{(l)})\subseteq \operatorname{soc}_{\bLam}(E)=\operatorname{soc}_{\bLam}(\iota(I))=\operatorname{soc}_{\Lambda_{l}}(\iota(I))$$ 
	and, consequently, $\operatorname{soc}_{\Lambda_{l}}(\iota(I))=\operatorname{soc}_{\Lambda_{l}}(E_{(l)})$.
	
	Since $I$ is injective in $\md(\Lambda_l)$, and $\iota(I)$ is a $\Lambda_l$-submodule of $E_{(l)}$, $\iota(I)$ is an internal direct summand of $E_{(l)}$ in $\md(\Lambda_l)$. Since $\operatorname{soc}_{\Lambda_{l}}(I)=\operatorname{soc}_{\Lambda_{l}}(E_{(l)})$, by \cite[Proposition 9.19]{AF92}
    we must have $\iota(I)=E_{(l)}$.
\end{proof}

By Lemma \ref{lemma:how-small-inj-sits-inside-bigger-inj},
\begin{equation}\label{eq:truncated-indec-injective-inside-big-injective}
\bigcup_{l\geq 1} I_{i,l} = \{m\in I_{i,\bLam}\mid \exists l>0 \ \text{such that} \ \mathfrak{m}^lm=0\},
\end{equation}
and from this we deduce that $I_{i,{\rm dis}} \coloneq D_2(e_i \Lambda)$ is an injective object of the category of locally nilpotent $\bLam$-modules, which is precisely ${\rm Dis}(\bLam)$ by Proposition \ref{prop:locnilpotent}.  
The proof of the following proposition is left in the reader's hands.



%


\begin{proposition}
The discrete-injective $\bLam$-module $I_{i,{\rm dis}}$ is indecomposable. Furthermore, every indecomposable discrete-injective $\bLam$-module is isomorphic to $I_{i,{\rm dis}}$ for a unique vertex $i\in Q_0$.
\end{proposition}

The next example shows that in general $I_{i,{\rm dis}}$ is not injective in $\Mod(\bLam)$. 

\begin{example}\label{ex:3.4}
Let $\bLam = \Gam = K\comp{Q}$ where $Q$ is the quiver
\[
\xymatrix{
1 & \ar[l]_\alpha 2 \ar@(ur,dr)[]^\beta
}.
\]
For $m \ge 0$ define $p_m \coloneq \alpha \beta^m$.
We obviously have $p_m \in \mathfrak{m}^{m+1}$.
The cyclic $\Gam$-module $\Gam p_m$ is then obviously
isomorphic to the simple $\Gam$-module $S_1$.
Next, one easily observes that
\[
U \coloneq \sum_{m \ge 0} \Gam p_m = \bigoplus_{m \ge 0} \Gam p_m,
\]
i.e. the sum of the submodules $\Gam p_m$ of $\Gam$ is a direct sum.
Let $\iota_U\colon U \to \Gam$ be the inclusion map, and let
$f \coloneq (f_m)_m\colon U \to S_1$, where
$f_m\colon \Gam p_m \to S_1$ is defined by $ap_m \mapsto ae_1$
with $a \in \Gam$.
Next, let $\iota_{S_1}\colon S_1 \to I_{1,\Gam}$ be a monomorphism.
By the injectivity of $I_{1,\Gam}$ there is a homomorphism $g\colon \Gam \to I_{1,\Gam}$
such that the diagram
\[
\xymatrix{
	U \ar[d]_{\iota_{S_1} \circ f}\ar[r]^{\iota_U} & \Gam \ar@{-->}[dl]^g
	\\
	I_{1,\Gam}
}
\]
commutes.
Let $v \coloneq g(1)$.
We have $g(a) = av$ for all $a \in \Gam$.
Observe that $0 \not= s \coloneq \iota_{S_1}(f(p_m)) = \iota_{S_1}(e_1)$ does not depend on $m$. 
We get 
\[
s = \iota_{S_1}(f(p_m)) = g(\iota_U(p_m)) = g(p_m) = p_mv. 
\]
It follows that $\mathfrak{m}^{m+1} v \not= 0$ for all $m \ge 0$.
In other words, $I_{1,\Gam}$ is not locally nilpotent.
Now Proposition~\ref{prop:locnilpotent} implies that $I_{1,\Gam}$ is not discrete. Hence $I_{1,\Gam}\neq I_{1,{\rm dis}}$.
\end{example}

In \cite[the paragraph before Lemma 2.2]{CLS15}, it is claimed that $I_i=I_{i,\operatorname{dis}}$, a fact that is in general false, see Example \ref{ex:3.4} above. 
As this was a peripheral claim, all other results in 
\cite{CLS15} are not affected by this mistake.

\section{Standard presentations} 

The first step to our main result is to derive from \cite{BK99} a standard projective presentation of finite-dimensional modules over $\Lam$, which will be achieved in this section. This construction then enables us to construct minimal projective presentations for all finite-dimensional modules over $\Lam$ in the next section.

\subsection{Completed enveloping algebra} 
\label{section:completed:enveloping:algebra}
To obtain a standard projective resolution for the algebra $\Lam$, we need to consider the \emph{completed enveloping algebra}. This is defined as 
\begin{equation*}
    \hat{\Gam}^e \coloneq \Gam \hat{\otimes}_K \Gam^{\text{op}} = \prod_{i,j \geq 0} \left(A^{\otimes_S^i}\otimes_K 
    (A^{\otimes_S^j})^{\text{op}}\right)
\end{equation*}
where $\Gam^{\text{op}}$ denotes the opposite algebra. This is again a basic pseudocompact algebra with primitive idempotents given by $e_i \otimes e_j$ for $i,j \in Q_0$. Equivalently, we can think of $\hat{\Gam}^e$ as a completed path algebra of a quiver. For this consider the quiver $Q^e$ with
\begin{equation*}
    Q_0^e = Q_0 \times Q_0 \text{ and } Q_1^e = (Q_1 \times Q_0) \cup (Q_0 \times Q_1),
\end{equation*}
and start and terminal vertices for $\alpha,\beta \in Q_1$ and $i \in Q_0$ as follows:
\begin{align*}
    s(\alpha,i) &= (s(\alpha), i), \quad t(\alpha,i)=(t(\alpha),i), \\
    s(i,\beta) &= (i,t(\beta)), \quad t(i,\beta)=(i,s(\beta)).
\end{align*}
Then we may think of $\hat{\Gam}^e$ as the completed path algebra of $Q^e$ modulo the ideal which is generated by the relations
\begin{equation*}
    (\beta,s(\alpha))(s(\beta),\alpha) - (t(\beta),\alpha)(\beta,t(\alpha)) \text{ for all } (\alpha,\beta) \in Q_1 \times Q_1.
\end{equation*} 
Note that this ideal is closed. In particular, the indecomposable projective $\hat{\Gam}^e$-modules are precisely of the form $P_{i,j}^e \coloneq \hat{\Gamma}^e(e_i \otimes e_j)$ for $(i,j) \in Q_0 \times Q_0$.

Moreover, $\hat{\Gam}^e = \Gam \hat{\otimes}_K \Gam^{\text{op}}$ acts naturally on $\Gam$ by $(p \otimes r) \cdot q \coloneq pqr$ with $p,q,r \in \Gam$. This turns $\Gam$ into a pseudocompact, (usually non-projective) $\hat{\Gam}^e$-module. It is convenient to observe that by acting with $\hat{\Gam}^e$ in the same way, we can identify $P_{i,j}^e \cong \Gam e_i \hat{\otimes} e_j \Gam$ as left $\hat{\Gam}^e$-modules. This observation allows us to identify $\Gam \hat{\otimes}_S \Gam$ with the projective $\hat{\Gam}^e$-module $\bigoplus_{i \in Q_0}P_{i,i}^e$. Similarly, we identify
\begin{equation*}
    (\Gam \otimes_S A) \hat{\otimes}_S \Gam = \Gam \hat{\otimes}_S (A \otimes_S \Gam) 
\end{equation*}
with
\begin{equation*}
    \bigoplus_{\alpha \in Q_1} P_{\alpha}^e \text{ where } P_{\alpha}^e \cong P_{t(\alpha),s(\alpha)}^e.
\end{equation*}
Indeed, by definition 
\[
A \coloneq \bigoplus_{\alpha \in Q_1}K\alpha\] and 
$e_{t(\alpha)} \alpha e_{s(\alpha)} = \alpha$ and so 
\[
\Gam \hat{\otimes}_S A \hat{\otimes}_S \Gam \cong \bigoplus_{\alpha \in Q_1} 
\Gam e_{t(\alpha)} \hat{\otimes} e_{s(\alpha)} \Gam.
\]
We write $\Gam \hat{\otimes}_S A \hat{\otimes}_S \Gam$ by slight abuse of notation for this projective $\hat{\Gam}^e$-module.

\subsection{Standard presentations}
Let us construct a standard projective presentation of finite-dimensional representations of $\Lam$. It seems to be well-known to specialists \cite[eq. 1.3]{BK99} that we have a standard projective $\hat{\Gam}^e$-resolution
\begin{equation}
    \label{eq:stand:proj:gamma:e:resolution}
    0 \rightarrow \Gam \hat{\otimes}_S A \hat{\otimes}_S \Gam \xrightarrow{d} \Gam \hat{\otimes}_S \Gamma \xrightarrow{e} \Gam \rightarrow 0,
\end{equation}
where the maps are defined as follows:
\begin{align*}
    d(p \otimes_S \alpha \otimes_S q) &= p\alpha \otimes_S q - p \otimes_S \alpha q \text{ for } p,q \in \Gam, \alpha \in A \text{ and}\\
    e(p \otimes_S q) &= pq \text{ for } p,q \in \Gam
\end{align*}
is the muliplication map. We claim that both $e$ and $d$ are continuous. Indeed, it suffices to show that the corresponding maps $e_0$ and $d_0$ for the non-completed tensor algebra $\Gam^0=T_S(A)$, that is, $KQ$ are continuous. The claim then follows by completion. The classical result for $\Gam^0$ follows by considering $\Gam^0$ naturally as a graded algebra. The corresponding maps $d_0$ and $e_0$ are then homogeneous of degree $0$, i.e. uniformly continuous with respect to the $\mathfrak{m}$-adic topology. This proves the claim.

To derive a projective presentation not only for $\Gam$ but also for $\Lam$ we follow \cite{DWZ10} in defining \emph{derivations}. As in \cite[eq. (8.6)]{DWZ10}, we have for each $\alpha \in Q_1$ a continuous linear map called \emph{derivation}
\begin{equation*}
    \Delta_{\alpha}\colon \Gam \rightarrow \Gam e_{t(\alpha)} \,\hat{\otimes}_S \,K\alpha \,\hat{\otimes}_S \,e_{s(\alpha)}\Gam,
\end{equation*}
which is defined on paths as 
\begin{equation*}
    \alpha_l\alpha_{l-1}\cdots \alpha_1 \mapsto \sum_{i: \alpha_i=\alpha} \alpha_l \cdots \alpha_{i+1} \otimes_S \alpha \otimes_S \alpha_{i-1} \cdots \alpha_1.
\end{equation*}
Later on, it will be convenient to have for each $\rho \in \Gam$ the notation
\begin{equation*}
    \Delta_\alpha(\rho) = \sum_{i} \Delta_{\alpha}^{(1,i)}(\rho) \otimes_S \alpha \otimes_S \Delta_{\alpha}^{(2,i)}(\rho).
\end{equation*}
In order for our result to work in arbitrary characteristic, we need to replace the cyclic derivatives used in  \cite{DWZ10} as follows. We define the continuous linear maps
\begin{align*}
    l_{\alpha} &\colon \Gam \rightarrow e_ {s(\alpha)}\Gam, \;\alpha_l\alpha_{l-1} \cdots \alpha_1 \mapsto \begin{cases}
        \alpha_{l-1}\cdots \alpha_2\alpha_1 &\text{if }  \alpha_l=\alpha, \\
        0 &\text{else},
    \end{cases}
    \; \text{ and} \\
    r_{\alpha} &\colon \Gam \rightarrow \Gam  e_{t(\alpha)}, \;\alpha_l\alpha_{l-1} \cdots \alpha_1 \mapsto \begin{cases}
        \alpha_l\alpha_{l-1}\cdots \alpha_2 &\text{if }  \alpha_1=\alpha, \\
        0 &\text{else}.
    \end{cases}
\end{align*}
We agree that $l_{\alpha}(\alpha) = e_{s(\alpha)}$, $r_{\alpha}(\alpha)=e_{t(\alpha)}$ and $l_\alpha(e_i)=0=r_{\alpha}(e_i)$ for all $i \in Q_0$. With this notation we have for each $\alpha \in Q_1$ and $\rho \in e_{t(\alpha)}\Gam$ obviously
\begin{align}
    \label{eq:delta:alpha:rho}
    \Delta_{\alpha}(\rho) - e_{t(\alpha)} \otimes_S \alpha \otimes_S l_{\alpha}(\rho) \in \mathfrak{m} \,\hat{\otimes}_S \,\alpha \, \hat{\otimes}_S \,e_{s(\alpha)}\Gam, \\
    \Delta_{\alpha}(\rho) - r_{\alpha}(\rho) \otimes_S \alpha \otimes_S e_{s(\alpha)} \in \Gam e_{t(\alpha)} \,\hat{\otimes}_S \,\alpha \, \hat{\otimes}_S \,\mathfrak{m}.
\end{align}

\begin{example}
    \label{example:cyc:differential}
    Consider the quiver $Q$ given by 
    \begin{equation*}
        \begin{tikzcd}[sep=scriptsize]
	        1 & 2
	        \arrow["\alpha", bend left, from=1-1, to=1-2]
	        \arrow["\beta", bend left, from=1-2, to=1-1]
    \end{tikzcd}.
    \end{equation*}
    Consider the path $\rho = (\beta\alpha)^n \in e_1 \Gam e_1$. We obtain
    \begin{equation*}
        \Delta_\alpha((\beta\alpha)^n) = \sum_{i=0}^{n-1} (\beta\alpha)^i\beta \otimes_S \alpha \otimes_S (\beta\alpha)^{n-i-1} \in \Gam e_1 \, \hat{\otimes}_S \, K\alpha \,\hat{\otimes}_S \,e_1\Gam.
    \end{equation*}
    Moreover, by definition,
    \begin{equation*}
    l_\alpha((\beta\alpha)^n)=0, \qquad r_\alpha((\beta\alpha)^n)=(\beta\alpha)^{n-1}\beta.
    \end{equation*}
    This example shows how we replace cyclic derivatives with a notion that still works in arbitrary characteristic.
\end{example}

\begin{lemma}
With the notation from Section~\ref{section:completed:path:algebras} we have an exact sequence of $\hat{\Gam}^e$-modules
\begin{equation}
    \label{lemma:Lamda:sequence}
    0 \rightarrow \mathfrak{J}/\mathfrak{J}^2 \xrightarrow{\overline{c}} \Lam \hat{\otimes}_S A \hat{\otimes}_S \Lam \xrightarrow{\overline{d}} \Lam \hat{\otimes}_S \Lam \xrightarrow{\overline{e}} \Lam\rightarrow 0,
\end{equation}   
where the maps $\overline{d}$ and $\overline{e}$ are induced by the corresponding maps in (\ref{eq:stand:proj:gamma:e:resolution}), and $\overline{c}$ is induced by the \emph{universal derivation}
\begin{equation*}
    c = (\Delta_\alpha)_{\alpha \in Q_1} \colon \Gam \rightarrow \Gam \hat{\otimes}_S A \hat{\otimes}_S \Gam.
\end{equation*}
\end{lemma}

\begin{proof}
    Note that the exact sequence in (\ref{eq:stand:proj:gamma:e:resolution}) splits as a sequence of $\Gam$-right modules and as a sequence of $\Gam$-left modules because $_{\Gam}\Gam$ and $\Gam_{\Gam}$ are projective $\Gam$-modules. Thus we can apply \cite[Proposition~3.2]{BK99} with $\Omega_1 = \Gam = \Omega_2$ and $H_1 = \mathfrak{J} = H_2$ to this exact sequence (with the notation from \cite[Proposition~3.2]{BK99}). Note that by definition $\Lam =  \Gam/\mathfrak{J}$.
\end{proof}

From this \emph{standard presentation} of $\Lam$ we also obtain a \emph{standard presentation} of all finite-dimensional modules $M$ over $\Lam$ as the following corollary shows:

\begin{corollary}\label{cor:standardpresentation}
    Let $M$ be a finite-dimensional representation of $\Lam$. Then we have the exact sequence
    \begin{equation}
        \label{eq:standard:presentation:module}
        \bigoplus_{\rho \in R} \Lam e_{t(\rho)} \otimes K\rho \otimes M(s(\rho)) \xrightarrow{c_M} \bigoplus_{\alpha \in Q_1} \Lam e_{t(\alpha)} \otimes K\alpha \otimes M(s(\alpha)) \xrightarrow{d_M} \bigoplus_{k \in Q_0} \Lam e_k \otimes M(k) \xrightarrow{e_M} M \rightarrow 0,
    \end{equation}
    where the maps are given by
    \begin{align*}
        e_M(p \otimes m) &= M(p)m \in M(l) \quad (p \in e_l\Lam e_k, m \in M(k)), \\
        d_M(p \otimes \alpha \otimes m) &= (p\alpha) \otimes m - p \otimes M(\alpha)m, \\
        c_M(p \otimes \rho \otimes m)_{\alpha} &= \sum_{i}p\Delta_{\alpha}^{(1, i)}(\rho) \otimes \alpha \otimes M(\Delta_{\alpha}^{(2,i)}(\rho))m.
    \end{align*}
\end{corollary}

\begin{proof}
    Consider the map $\tilde{c}$ which is defined to be the composition of the canonical surjection
    \begin{equation*}
        \bigoplus_{\rho \in R}\Lam \, \hat{\otimes}_S \,K\rho\, \hat{\otimes}_S \,\Lam \rightarrow \mathfrak{J}/\mathfrak{J}^2, \; p \otimes_S \rho \otimes_S q \mapsto p\rho q + \mathfrak{J}^2
    \end{equation*}
    with the map $\overline{c}$ from Lemma~\ref{lemma:Lamda:sequence}. Then we obtain from Lemma~\ref{lemma:Lamda:sequence} trivially an exact sequence
    \begin{equation}
        \label{eq:standard:presentation:proof}
        \bigoplus_{\rho \in R}\Lam\, \hat{\otimes}_S \,K\rho \,\hat{\otimes}_S \,\Lam \xrightarrow{\tilde{c}} \Lam \hat{\otimes} A \hat{\otimes} \Lam \xrightarrow{\overline{d}} \Lam \hat{\otimes} \Lam \xrightarrow{\overline{e}} \Lam \rightarrow 0.
    \end{equation}
    Next, since $\Lam_\Lam$ is a projective $\Lam$-module, the map $\overline{e}$ splits as a morphism of $\Lam$-right modules. Thus, since the functor $ - \otimes_\Lam M$ is right exact, the sequence in (\ref{eq:standard:presentation:proof}) remains exact under this functor. It remains to show that $e_M= \overline{e} \otimes_{\Lam} M$, $d_M= \overline{d} \otimes_{\Lam} M$ and $c_M= \tilde{c} \otimes_{\Lam} M$. This follows by looking at how $c,e$ and $d$ were defined.
\end{proof}

Note that for $\Lambda$ the path algebra of an acyclic quiver $Q$, the projective presentation in Corollary~\ref{cor:standardpresentation}
turns into Crawley-Boevey's \emph{standard resolution} from \cite[p. 7]{CB92}.

\begin{example}
    Let $\Lam = K\comp{Q}/\mathfrak{J}$ with quiver $Q$ given by
    \begin{equation*}
        \begin{tikzcd}
        1 \arrow[r, "\alpha"] & 2 \arrow[r, "\beta"] & 3
        \end{tikzcd}
    \end{equation*}
    and $\mathfrak{J} = \gen{\beta\alpha}$. Consider $M=P_1 \oplus S_3$, direct sum of a projective and a simple module. Then the sequence in (\ref{eq:standard:presentation:module}) becomes
    \begin{equation*}
        \Lam e_3 \xrightarrow{c_M} (\Lam e_2 \oplus \Lam e_3) \xrightarrow{d_M} (\Lam e_1 \oplus \Lam e_2 \oplus \Lam e_3) \xrightarrow{e_M} M \rightarrow 0.
    \end{equation*}
    where the maps are given by the explicit formulas above.
\end{example}

\section{Minimal presentations derived from standard complexes} \label{sec:main}
In this section, we will refine the projective presentation 
(\ref{eq:standard:presentation:module}) to a minimal projective presentation. 

\subsection{Standard complexes} We generalize Derksen--Weyman--Zelevinsky's 
complex (\ref{eq:complex:jacobian}) as follows:

\begin{definition}
    \label{definition:complex:vector:spaces}
    Let $M$ be a finite-dimensional representation of $\Lam$ and $k \in Q_0$. Then we have the following complex of vector spaces
    \begin{equation}
        \label{eq:complex:vector:space}
M_{\text{rel}}'(k)  \xrightarrow{M(\gamma_k')} M_{\text{in}}(k) \xrightarrow{M(\alpha_k)} M(k) \to 0
    \end{equation}
    where
    \begin{align}
        M_{\text{in}}(k) &= \bigoplus_{\alpha \in Q_1 : \; t(\alpha)=k} M_{\alpha}  &\text{with } M_{\alpha} \coloneq M(s(\alpha)), \label{eq:M:in}\\
        M_{\text{rel}}'(k) &= \bigoplus_{\rho \in R : \; t(\rho)=k} M_{\rho}'  &\text{with } M_{\rho}' \coloneq M(s(\rho)), \label{eq:M:rel}\\
        M(\alpha_k)(m_{\alpha}) &=M(\alpha)m_\alpha, \label{eq:M:alpha:k}\\
        M(\gamma_k')_{\alpha}(m_{\rho}') &= M(l_{\alpha}(\rho))m_{\rho}'. \label{eq:M:gamma:k}
    \end{align}
    Note that $M(\alpha_k) \circ M(\gamma_k')=0$ since for $\rho \in R$ with $t(\rho)=k$ we have
    \begin{equation*}
        \sum_{\alpha \in Q_1: \; t(\alpha)=k} \alpha \cdot l_\alpha(\rho) = \rho.
    \end{equation*}
We call (\ref{definition:complex:vector:spaces}) a \emph{standard complex} for $M$.
\end{definition}

\begin{example}
    Let $Q$ be the quiver
    \[
    \xymatrix@-3ex{
    & 2 \ar[ddr]^\beta
    \\
    &
    \\
    1 \ar[uur]^\alpha && 3 \ar[ll]^\gamma
    }
    \]
    and let $\Gam = K\comp{Q}$.
    Consider the relations
    \begin{equation*}
        \rho_1 = \gamma\beta \in e_1 \Gam e_2, \qquad \rho_2 = \alpha\gamma \in e_2 \Gamma e_3, \qquad \rho_3 = \beta\alpha \in e_3 \Gam e_1.
    \end{equation*}
    Let $\Lam = \Gam/\gen{\rho_1,\rho_2,\rho_3}$.

    Let $M = P_1 \oplus P_2$ be the direct sum of the two indecomposable projective modules associated to the vertices $1$ and $2$, respectively.
    We have $M(1) = Ke_1$, $M(2) = K\gen{e_2,\alpha}$, $M(3) = K\gen{e_3,\beta}$, and
    \begin{equation*}
        M_{\text{in}}(i)=M(i-1),
        \qquad 
        M_{\text{rel}}'(i)=M(i-2),
    \end{equation*}
    where the indices are taken modulo $3$. 
We get the maps
\begin{align*}
M(\alpha_2) = M(\alpha):M(1)&\to M(2)
\\
e_1 &\mapsto \alpha,
\\
M(\alpha_3) = M(\beta): M(2) &\to M(3)
\\
e_2 &\mapsto \beta
\\
\alpha &\mapsto 0.
\end{align*}
    
    Next, we compute the maps $l_\alpha(\rho)$. We have
    \begin{align*}
        l_\alpha(\rho_1)=0,\qquad l_\beta(\rho_1)=0, \qquad l_\gamma(\rho_1)=\beta, \\
        l_\alpha(\rho_2)=\gamma,\qquad l_\beta(\rho_2)=0, \qquad l_\gamma(\rho_2)=0, \\
        l_\alpha(\rho_3)=0,\qquad l_\beta(\rho_3)=\alpha, \qquad l_\gamma(\rho_3)=0.
    \end{align*}
    Hence we get for example the map
    \begin{align*}
M(\gamma_3') = M(\alpha): M_{\rm rel}'(3) = M(1) &\to M_{\rm in}(3) = M(2)
\\
e_1 &\mapsto \alpha
    \end{align*}
    Hence for example the map $M(\gamma'_3):M_{\text{rel}}'(3)=M(1)\to M_{\text{in}}(3)=M(2)$ is given by
    \begin{equation*}
        M(\gamma'_3)(e_1)=M(\alpha)e_1 = \alpha.
    \end{equation*}
    Finally, we compute
    \begin{equation*}
        M(\alpha_3) \circ M(\gamma'_3) (e_1) = M(\alpha_3)(\alpha)=M(\beta)\alpha=0.
    \end{equation*}
    Thus the sequence is indeed a complex.    
\end{example}

\begin{remark}
    \label{remark:dim:min:proj}
    Denote by $Q_l$ the set of paths of length $l$ in $Q$. Then $R_l \coloneq R \cup Q_l$ is a system of generators of $\mathfrak{J}_l \coloneq \mathfrak{J} + \mathfrak{m}^l$. Similarly, we write $\overline{R} = R \cup R'$ for an appropriate set of generators of the closure $\overline{\mathfrak{J}}$. Since $M$ is in particular a finite-dimensional representation of $\Lam$, we can define 
        \[
        n(M)\coloneq\min\Set{l \geq 0 \mid \mathfrak{m}^lM=0} < \infty.
        \]
        Moreover, of course we have $\mathfrak{m}^lM=0$ for $l\geq n(M)$. Thus, $M$ descends in this case to a representation of the truncation $\Lam_l$. In particular, $M$ may be viewed as a representation of $\bLam$. Let us write 
        \begin{equation*}
            M(\gamma_{k,l}')\colon M_{\text{rel},l}'(k) \rightarrow M_{\text{in}}(k)
        \end{equation*}
        for the map $M(\gamma_k')$ from the above definition (\ref{eq:M:gamma:k}), but associated to the set of relations $R_l$ for $l \geq n(M)$. Analogously, we denote $M(\overline{\gamma}_k')$ for the map associated to the set of relations $\overline{R}$. By definition of $n(M)$, it is immediate that
        \begin{equation}\label{eq:gammatruncation}
            \im(M(\gamma_{k,l}')) = \im(M(\gamma_k')) = \im(M(\overline{\gamma}_k'))
        \end{equation}
        for $l \geq n(M) + 1$ since relations of longer length do not affect the image of $M(\gamma_k')$.
\end{remark}

\subsection{Minimal projective presentations}\label{subsec:5.2} The generalized complex of vector spaces in (\ref{eq:complex:vector:space}) allows us to prove our main result, describing a standard minimal projective presentation of finite-dimensional modules over $\Lam$, generalizing the one for Jacobian algebras presented in \cite{DWZ10}. With the notations from Definition~\ref{definition:complex:vector:spaces} we introduce the following auxiliary vector spaces for each $k \in Q_0$. We define subspaces $M^{(0)}(k)$ and $M^{(1)}(k)$ of $M(k)$, as well as for $j = 0,1,2$ subspaces $M_{\text{in}}^{(j)}(k)$ of $M_{\text{in}}(k)$ such that
\begin{align*}
    M^{(1)}(k) &\coloneq \im(M(\alpha_k)), \\
    M^{(0)}(k) \oplus M^{(1)}(k) &=M(k), \\
    M_{\text{in}}^{(2)}(k) &\coloneq \im(M(\gamma_k')), \\
    M_{\text{in}}^{(1)}(k) \oplus M_{\text{in}}^{(2)}(k) &= \ker(M(\alpha_k)), \\
    M_{\text{in}}^{(0)}(k) \oplus \ker(M(\alpha_k)) &= M_{\text{in}}(k),
\end{align*}
similar to \cite[eq. 10.15]{DWZ10}. Then we obtain from (\ref{eq:standard:presentation:module}) the following projective presentation of $M$:
\begin{equation*}
    P_M \xrightarrow{d_M} K_M \xrightarrow{e_M} M \rightarrow 0,
\end{equation*}
where we have set
\begin{align*}
    K_M  &\coloneq \bigoplus_{k \in Q_0} \Lam e_k \otimes M(k), \\
    P_M &\coloneq \bigoplus_{k \in Q_0} \Lam e_k \otimes M_{\text{in}}(k) \cong \bigoplus_{\alpha \in Q_1} \Lam e_{t(\alpha)} \otimes K\alpha \otimes M(s(\alpha)).
\end{align*}
Here, the maps are as in (\ref{eq:standard:presentation:module}). With a slight abuse of notation we may rewrite $d_M = d_1 + d_2$ such that for $p \otimes m \in \Lam e_k \otimes M_{\text{in}}(k)$ we have according to the definition of $M_{\text{in}}(k)$ that
\begin{align*}
    d_1(p \otimes m) &\coloneq \sum_{\alpha \in Q_1: \; t(\alpha)=k} p\alpha \otimes \pr_{\alpha}(m), \\
    d_2(p \otimes m) &\coloneq - p \otimes M(\alpha_k)m.
\end{align*}
Here, $\pr_{\alpha} \colon M_{\text{in}}(k) \rightarrow M_\alpha$ is the canonical projection, where we recall $M_{\alpha}\coloneq M(s(\alpha))$. We finally set
\begin{align*}
    K_M^{(j)} &\coloneq \bigoplus_{k \in Q_0} \Lam e_k \otimes M^{(j)}(k) \text{ for } j = 0,1, \\
    P_M^{(j)} &\coloneq \bigoplus_{k \in Q_0} \Lam e_k \otimes M_{\text{in}}^{(j)}(k) \text{ for } j = 0,1,2.
\end{align*}
Thus, $K_M = K_M^{(0)} \oplus K_M^{(1)}$ and $P_M = P_M^{(0)} \oplus P_M^{(1)} \oplus P_M^{(2)}$ as projective $\Lam$-modules.

The following result is a generalization of \cite[Proposition~10.4]{DWZ10} from certain finite-dimensional Jacobian algebras to basic algebras which are topologically finitely generated. Our result implies in particular, that the statement of \cite[Proposition~10.4]{DWZ10} holds for arbitrary Jacobian algebras with reduced potential. Our proof follows roughly the strategy of Derksen-Weyman-Zelevinsky in \cite[Proposition~10.4]{DWZ10}; however, we use the non-commutative Nakayama Lemma in place of the nilpotency of the Jacobson radical. 

\begin{theorem}\label{thm:main2}
    Let $\mathfrak{J}\subseteq \mathfrak{m}^2$ be an ideal of $\Gam = K\comp{Q}$, and let $M$ be a finite-dimensional representation of $\Lam = \Gam/\mathfrak{J}$. With the above notation, the following hold:
    \begin{enumerate}
        \item The restriction $e_M^{(0)}$ of $e_M$ to $K_M^{(0)}$ is a minimal projective cover of $M$.
        \item We have $K_M = K_M^{(0)} \oplus d_M(P_M^{(0)})$. Thus, the canonical projection $\pr_K^{(0)} \colon K_M \rightarrow K_M^{(0)}$ along $d_M(P^{(0)}_M)$ is defined.
        \item With $\overline{d}_M\colon P_M^{(1)} \rightarrow K_M^{(0)}, p \mapsto \pr_K^{(0)}(d_M(p))$ holds $\im(\overline{d}_M)=\ker(e_M^{(0)})$.
        \item We have $\ker(\overline{d}_M) \subseteq \mathfrak{m}P_M^{(1)}$.
    \end{enumerate}
    In particular,
    \begin{equation*}
        P_M^{(1)} \xrightarrow{\overline{d}_M} K_M^{(0)} \xrightarrow{e_M^{(0)}} M \rightarrow 0
    \end{equation*}
    is a minimal projective presentation of $M$.
\end{theorem}

\begin{proof}
    \begin{enumerate}
        \item We claim that $M/\mathfrak{m}M \cong \bigoplus_{k \in Q_0} S_k \otimes M^{(0)}(k)$. Indeed, we can rewrite $M/\mathfrak{m}M \cong \bigoplus_{k \in Q_0} S_k \otimes M(k)$, where $S_k = \Lam e_k /\mathfrak{m} e_k$ is the simple module at vertex $k$. Now we can decompose $M(k) = M^{(0)}(k) \oplus M^{(1)}(k)$ with $ M^{(1)}(k)=\im(M(\alpha_k))$. Since $\mathfrak{m}$ is generated by arrows $\alpha \in Q_1$ and acts on $M(k)$ exactly via the maps $M(\alpha)$, we obtain $(\mathfrak{m}M)(k)= M^{(1)}(k)$.
        
        Moreover, we calculate
        \begin{equation*}
            K_M^{(0)}/\mathfrak{m}K_M^{(0)} = \bigoplus_{k \in Q_0} (\Lam e_k/ \mathfrak{m}\Lam e_k) \otimes M^{(0)}(k) \cong \bigoplus_{k \in Q_0} S_k \otimes M^{(0)}(k).
        \end{equation*}
        Thus, $e_M$ induces an isomorphism $K_M^{(0)}/(\mathfrak{m}K_M^{(0)}) \rightarrow M/\mathfrak{m}M$, since it is the identity. Then, our claim follows from the 
        Nakayama Lemma~\ref{Nakayama:lemma}(ii).
        \item We observe that $d_2(P_M^{(j)})=0$ for $j = 1,2$, because of $M_{\text{in}}^{(1)}(k) \oplus M_{\text{in}}^{(2)}(k) = \ker(M(\alpha_k))$.
        Moreover, $d_2$ induces an isomorphism $d_0: P_M^{(0)} \rightarrow K_M^{(1)}$ by restriction, as $M(\alpha_k)\colon M_{\text{in}}^{(0)}(k) \rightarrow M^{(1)}(k)=\im(M(\alpha_k))$ is an isomorphism.
        We have $d_1(P_M)\subseteq \mathfrak{m}K_M$ by definition of $d_1$. Let us show that the endomorphism
        \begin{equation*}
            f \colon K_M=K_M^{(0)} \oplus K_M^{(1)} \rightarrow K_M, (k_0,k_1) \mapsto k_0 + d_M(d_0^{-1}(k_1))
        \end{equation*}
        is surjective. Indeed, we can apply part (i) of the Nakayama Lemma with $N\coloneq K_M, N'\coloneq d_1(P_M)$ and $N'' \coloneq \im(f)$. Then, $N' \subseteq \mathfrak{m}N$ is clear, and for $N' + N'' =N$, we calculate: Take $(x_0,x_1) \in K_M= K_M^{(0)} \oplus K_M^{(1)}$ and write $x_1=d_0(p)$ with $p \in P_M^{(0)}$. Then $(x_0,x_1)=(x_0,0) + (0, x_1) = (x_0,0) + (d_M(p)-d_1(p))$, but $(x_0,0)=f(x_0,0), d_M(p)=f(0,x_1) \in \im(f)$ and $d_1(p) \in d_1(P_M)=N'$ proving 
        $K_M \subseteq d_1(P_M) + \im(f)$. By the Nakayama Lemma, $\im(f)=K_M$, so $f$ is surjective.

        Since $K_M$ is a projective module, $\ker(f)$ is a direct summand of $K_M = K_M^{(0)} \oplus K_M^{(1)}$. By the Krull-Remak-Schmidt-property (see Section~\ref{section:projective:modules}) for projective $\Lam$-modules applied to $K_M =\ker(f) \oplus K_M$, we conclude $\ker(f)=0$. Our claim follows since the automorphism $f$ takes $K_M^{(0)}$ to $K_M^{(0)}$ and $K_M^{(1)}$ to $d_M(P_M^{(0)})$.
        \item We have
        \begin{equation}
            \label{eq:ker:e:M:0}
            \ker(e_M^{(0)}) = \ker(e_M) \cap K_M^{(0)} = \im(d_M) \cap K_M^{(0)} = \pr_K^{(0)}(d_M(P_M^{(1)}\oplus P_M^{(2)})).
        \end{equation}
        In fact, the first two equalities are clear from our constructions. For the last step, note that since $\pr_K^{(0)}$ is the projection along $d_M(P_M^{(0)})$, it annihilates $d_M(p_0)$, and therefore
        \begin{equation*}
           \pr_K^{(0)}(d_M(p_0+p_1+p_2)) = \pr_K^{(0)}(d_M(p_1+p_2)).
        \end{equation*}
        Moreover, write $d_M(p_1 + p_2) = k_0 + d_M(p_0)$ for some $k_0 \in K_M^{(0)}$ and $p_0 \in P_M^{(0)}$. In other words,
        \begin{equation*}
            k_0 = \pr_K^{(0)}(d_M(p_1 + p_2)) = d_M(-p_0+p_1+p_2),
        \end{equation*}
        since $k_0 = d_M(p_1+p_2) -  d_M(p_0)$. Thus, $e_M^{(0)}(k_0)=e_M(k_0) = e_M(d_M(-p_0+p_1+p_2))=0$.
        In particular, $\ker(e_M^{(0)})=\pr_K^{(0)}(d_M(P_M^{(1)}\oplus P_M^{(2)}))$ is finitely generated, since images of finitely generated modules are finitely generated. Then (iii) again follows from the Nakayama Lemma applied to $N \coloneq \ker(e_M^{(0)}), N'\coloneq \pr_K^{(0)}(d_M(P_M^{(2)}))$ and $N'' \coloneq \im(\overline{d}_M)$. For this, note that $N=N'+N''$ which is just a rewriting of (\ref{eq:ker:e:M:0}). Moreover, 
        \begin{equation*}
             \pr_K^{(0)}(d_M(P_M^{(2)})) = \pr_K^{(0)}(d_1(P_M^{(2)})+ d_2(P_M^{(2)}))= \pr_K^{(0)}(d_1(P_M^{(2)}))\subseteq \mathfrak{m}\ker(e_M^{(0)})
        \end{equation*} 
        by definition of $d_1$ and (\ref{eq:ker:e:M:0}).
        \item Observe that $P_M^{(2)}$ is generated by elements of the form
        \begin{equation*}
             e_k \otimes M(\gamma_k')m \quad (k \in Q_0, m \in M_{\text{rel}}'(k)).
        \end{equation*}
        In view of the definition of $M(\gamma_k')$ in (\ref{eq:M:gamma:k}) we have 
        \begin{equation*}
             e_k\otimes M(\gamma_k')m = \sum_{\alpha \in Q_1 : \; t(\alpha)=k} e_k \otimes M(l_{\alpha}(\rho))m
        \end{equation*}
         and by definition 
         \begin{equation*}
            (c_M(e_k \otimes m))_{\alpha} = \sum_{i}\Delta_{\alpha}^{(1, i)}(\rho) \otimes M(\Delta_{\alpha}^{(2,i)}(\rho))m.
        \end{equation*}
        Because of (\ref{eq:delta:alpha:rho}), we have
        \begin{equation*}
             e_k \otimes M(\gamma_k')m - c_M(e_k \otimes m) \in \mathfrak{m}P_M.
        \end{equation*}    
        This implies that 
        \begin{equation*}
            \ker(d_M) = \im(c_M) \subseteq \mathfrak{m}P_M^{(0)} \oplus \mathfrak{m}P_M^{(1)} \oplus P_M^{(2)}
        \end{equation*}
        since $e_k\otimes M(\gamma_k')m \in P_M^{(2)}$. Now, by construction,
        \begin{equation*}
            \ker(\overline{d}_M)=(\ker(d_M)+P_M^{(0)}) \cap P_M^{(1)} \subseteq (P_M^{(0)} \oplus \mathfrak{m}P_M^{(1)} \oplus P_M^{(2)}) \cap P_M^{(1)} = \mathfrak{m}P_M^{(1)}.
        \end{equation*}
        The first equality follows since for  $p_1 \in P_M^{(1)}$ we have $p_1 \in \ker(\overline{d}_M)$ iff $\pr_K^{(0)}(d_M(p_1))=0$ iff $p_1 \in \ker(d_M)+P_M^{(0)}$. This implies that $\overline{d}_M$ is radical, i.e. we obtain a minimal projective presentation of $M$.
    \end{enumerate}\end{proof}

\begin{theorem}\label{thm:main3}
    Let $\Lam = \Gam/\mathfrak{J}$ for an ideal $\mathfrak{J}$ which is generated by a set $R$ of relations, and let 
    $\bLam \coloneq \Gam/\overline{\mathfrak{J}}$ as in Section~\ref{section:completed:path:algebras}. 
    Let $M \in {\rm mod}(\Lam)$.
    Then there is a minimal projective presentation
    \begin{equation*}
        \bigoplus_{k \in Q_0} (\Lam e_k)^{\dim(\ker(M(\alpha_k))/\im(M(\gamma_k')))} \rightarrow \bigoplus_{k \in Q_0} (\Lam e_k)^{\dim\coker(M(\alpha_k))} \rightarrow M \rightarrow 0
    \end{equation*}
    in $\Mod(\bLam)$,
    where the maps $M(\gamma_k')$ are defined in terms of the set $R$ of generators of $\mathfrak{J}$ as in (\ref{eq:M:gamma:k}).
    Furthermore, the $\Lam$-module $M$ can be seen as a $\bLam$-module, and has a minimal projective presentation
    \begin{equation*}
        \bigoplus_{k \in Q_0} (\bLam e_k)^{\dim(\ker(M(\alpha_k))/\im(M(\gamma_k')))} \rightarrow \bigoplus_{k \in Q_0} (\bLam e_k)^{\dim\coker(M(\alpha_k))} \rightarrow M \rightarrow 0.
    \end{equation*}
    in $\Mod(\bLam)$.
\end{theorem}

\begin{proof}
By Theorem~\ref{thm:main2} there is a minimal projective presentation
\[
P_M^{(1)} \xrightarrow{\overline{d}_M} K_M^{(0)} \xrightarrow{e_M^{(0)}} M \to 0
\]
where
\begin{align*}
P_M^{(1)} &= \bigoplus_{k \in Q_0} \Lam e_k \otimes M_{\rm in}^{(1)}(k),
&
M_{\rm in}^{(1)}(k) \oplus \im(M(\gamma_k')) &= \ker(M(\alpha_k)),
\\
K_M^{(0)} &= \bigoplus_{k \in Q_0} \Lam e_k \otimes M^{(0)}(k),
    &
    M^{(0)}(k) \oplus \im(M(\alpha_k)) &= M(k).
\end{align*}
This shows the first part of the theorem.
Since every finite-dimensional $\Lam$-module $M$ factors through $\bLam = \Gam/\overline{\mathfrak{J}}$ and the maps $M(\alpha_k), M(\gamma_k')$ only depend on finite-length relations (see Remark~\ref{remark:dim:min:proj} (in particular (\ref{eq:gammatruncation})) for more details), 
the second part of the theorem holds.    
\end{proof}


\subsection{Minimal injective presentations} 
In the last section, we have constructed a minimal projective presentation. We will now dualize this similar to \cite{DWZ10}, which will give us a statement analogous to (\ref{eq:min:inj:copres:jacobian}). 
The following definition is dual to Definition~\ref{definition:complex:vector:spaces}.

\begin{definition}
    \label{definition:complex:vector:spaces:dual}
    Let $M$ be a finite-dimensional representation of $\Lam$ and $k \in Q_0$. Then we have the following complex of vector spaces
    \begin{equation}
        \label{eq:complex:vector:space:dual}
        0 \rightarrow M(k) \xrightarrow{M(\beta_k)} M_{\text{out}}(k) \xrightarrow{M(\gamma_k'')}
        M_{\text{rel}}''(k)  
    \end{equation}
    where
    \begin{align}
        M_{\text{out}}(k) &= \bigoplus_{\beta \in Q_1 : \; s(\beta)=k} M_{\beta}  &\text{with } M_{\beta} \coloneq M(t(\beta)), \label{eq:M:out}\\
        M_{\text{rel}}''(k) &= \prod_{\rho \in R : \; s(\rho)=k} M_{\rho}''  &\text{with } M_{\rho}'' \coloneq M(t(\rho)), \label{eq:M:rel:dash}\\
        M(\beta_k)(m) &=M(\beta)m, \label{eq:M:beta:k}\\
        M(\gamma_k'')_{\rho}(m_{\beta}) &= M(r_{\beta}(\rho))m_{\beta}. \label{eq:M:gamma:k:dash}
    \end{align}
    Note that $M(\gamma''_k) \circ M(\beta_k)=0$ since for $\rho \in R$ with $s(\rho)=k$ we have
    \begin{equation*}
        \sum_{\beta \in Q_1: \; s(\beta)=k} r_\beta(\rho) \cdot \beta = \rho.
    \end{equation*}
The complex (\ref{definition:complex:vector:spaces:dual}) is also called a \emph{standard complex} for $M$.
\end{definition}

The following observations allow us to obtain a minimal injective presentation dual to the minimal projective presentation. Recall that the algebra $\bLam = \Gam/\overline{\mathfrak{J}}$ is pseudocompact. Thus,
we have a duality
\begin{equation*}
    D_2 \colon {\rm PC}(\bLam^{\op}) \rightarrow \Dis(\bLam)
\end{equation*}
between the category of pseudocompact $\bLam^{\op}$-modules with continuous homomorphisms on one side and discrete 
$\bLam$-modules on the other side. Here $D_2(M) = {\rm hom}_K(M,K)$ consists of the \emph{continuous linear forms} on $M$. For more details and definitions see for example \cite[Section 2]{SimsonDaniel2001Ccpa}. 
Under this duality, indecomposable pseudocompact projective modules get send to indecomposable discrete-injective modules. 

Note that each finitely generated projective $\bLam$-module is pseudocompact.

On the other hand, the finitely cogenerated injective $\bLam$-modules 
are in general not discrete.

We know from Section~\ref{subsec:injectives} that
the indecomposable discrete-injective $\bLam$-modules are
 \[
 D_2(e_k\bLam) = \bigcup_{l \ge 0} D(e_k \Lam_l).
 \]
 where $D$ is the usual duality $\hom_K(-,K)$.
Note also that $D(D_2(e_k \bLam)) \cong e_k\bLam$.
This allows us to dualize Theorem~\ref{thm:main3}:

\begin{theorem}\label{thm:main4}
     Let $\Lam = \Gam/\mathfrak{J}$ and $\bLam \coloneq \Gam/\overline{\mathfrak{J}}$ as in Section~\ref{section:completed:path:algebras}. Then each finite-dimensional representation $M$ of $\bLam$ has a minimal discrete-injective presentation
     \begin{equation*}
        0 \rightarrow M \rightarrow \bigoplus_{k \in Q_0} D_2(e_k\bLam)^{\dim\ker(M(\beta_k))} \rightarrow \bigoplus_{k \in Q_0} D_2(e_k\bLam)^{\dim(\ker(M(\gamma_k''))/\im(M(\beta_k)))}.
    \end{equation*}
\end{theorem}

\begin{proof}
    Consider $M'\coloneq \hom_K(M,K)$ as a finite-dimensional representation of $\bLam^{\rm op}$. By Theorem~\ref{thm:main3}, it has a finite minimal projective presentation. If we apply to it the duality $D_2$ between pseudocompact $\bLam^{\rm op}$-modules and discrete 
    $\bLam$-modules we obtain the desired minimal discrete-injective presentation of $M$.
\end{proof}

Note that Theorem~\ref{thm:main4} generalizes \cite[Remark~10.8]{DWZ10}.

\subsection{Standard resolutions of simple modules and standard complexes}
The aim of this section is to give a conceptual explanation of the standard complexes
(\ref{eq:complex:vector:space}) and (\ref{eq:complex:vector:space:dual}) for $M \in {\rm mod}(\Lam)$.
Namely, they arise naturally from the standard discrete-injective resolution 
(resp. the standard projective resolution) of the simple module $S_k$.

    For the simple module $S_k = \Lam e_k / \mathfrak{m} e_k$ with $k \in Q_0$ the sequence 
    (\ref{eq:standard:presentation:module}) specializes to 
    \begin{equation}
        \label{eq:standard:presentation:simple}
        \bigoplus_{\rho \in R: \; s(\rho)=k} \Lam e_{t(\rho)} \otimes K\rho \otimes Ke_k \xrightarrow{c_k} \bigoplus_{\alpha \in Q_1: \; s(\alpha)=k} \Lam e_{t(\alpha)} \otimes K\alpha \otimes Ke_k \xrightarrow{d_k} \Lam e_k \xrightarrow{p_k} S_k \rightarrow 0,
    \end{equation}
    where $p_k$ is the canonical projection,
    \begin{align*}
        d_k(p \otimes \alpha \otimes e_k) &= p \cdot \alpha \cdot e_k, \\
        c_k(p \otimes \rho \otimes e_k)_{\alpha} &= p \cdot r_\alpha(\rho) \otimes \alpha \otimes e_k.
    \end{align*}
    We can recover the standard complex (\ref{eq:complex:vector:space:dual}) for $M$ if we apply the functor $\hom_{\Lam}(-,M)$ to (\ref{eq:standard:presentation:simple}). 

Assume now that the set $R$ of relations is finite.
Then (\ref{eq:standard:presentation:simple}) descends to the beginning of a projective
resolution
\begin{equation}
        \label{eq:standard:presentation:simplebar}
        \bigoplus_{\rho \in R: \; s(\rho)=k} \bLam e_{t(\rho)} \otimes K\rho \otimes Ke_k \xrightarrow{c_k} \bigoplus_{\alpha \in Q_1: \; s(\alpha)=k} \bLam e_{t(\alpha)} \otimes K\alpha \otimes Ke_k \xrightarrow{d_k} \bLam e_k \xrightarrow{p_k} S_k \rightarrow 0
    \end{equation}
of $S_k$ in $\Mod(\bLam)$.
Since the modules $\bLam e_i$ are pseudocompact, this is also the beginning of a projective resolution
in $PC(\bLam)$.
By duality we obtain the beginning of an injective resolution
\begin{equation}
    \label{eq:standard:copresentation:simple}
    0 \rightarrow S_k \xrightarrow{i_k} D_2(e_k\bLam) \xrightarrow{D_2(d_k')} \bigoplus_{\alpha \in Q_1: \; t(\alpha)=k} D_2(e_{s(\alpha)}\bLam) \xrightarrow{D_2(c_k')} \bigoplus_{\rho \in R: \; t(\rho)=k} D_2(e_{s(\rho)}\bLam)
\end{equation}
of $S_k$ in ${\rm Dis}(\bLam)$.
Here $d_k'$ resp. $c_k'$ are the maps corresponding to the beginning of the projective resolution of the simple $\bLam^{\text{op}}$-module $S_k'$.

If we apply the functor $D\hom(M,-)$ to (\ref{eq:standard:copresentation:simple}), we obtain the standard complex (\ref{eq:complex:vector:space}) for $M$. Indeed, for $M \in {\rm mod}(\Lam)$, one has natural identifications
        \begin{equation*}
            D\hom_{\bLam}(M,D_2(e_i\bLam)) \cong e_iM \cong M(i),
        \end{equation*}
        which proves that the vector spaces coincide. Spelling out what dualizing and applying the functor $D\hom(M,-)$ do on the maps $d_k'$ and $c_k'$ yields the result. 

        \section{Applications to Jacobian algebras defined by hyperpotentials}

\subsection{Minimal presentations for Jacobian algebras} 
\label{section:jacobian:algebras}
As before, let $Q$ be a quiver.
A \emph{potential} $W$ for $Q$ is a (possibly infinite) linear combination of cycles in $Q$.
Such a potential is \emph{reduced} if all of these cycles have length at least $3$.

Assume now that $W$ is reduced.
In this situation, we set \[\rho_\alpha \coloneq \partial_\alpha(W)\] for each $\alpha \in Q_1$, where $\partial_\alpha$ denotes the \emph{cyclic derivative} defined by
\begin{equation*}
    \partial_\alpha(\alpha_1 \cdots \alpha_m) \coloneq \sum_{\substack{1 \leq p \leq m\\
                  \alpha_p = \alpha}}
        \alpha_{p+1}\cdots \alpha_m\alpha_1 \cdots \alpha_{p-1}
\end{equation*}
for each cycle $\alpha_1 \cdots \alpha_m$ appearing $W$.
(In particular, we have $m \ge 3$.)
We refer to
\cite[Section~4]{DWZ10} for more details on Jacobian algebras. 
With this notation, let
\[
R(W) \coloneq \Set{\rho_\alpha \mid \alpha \in Q_1},
\]
and let $I(W)$ be the ideal of 
$\Gam = K\comp{Q}$ which is generated by the elements of $R(W)$. Let $\Lam \coloneq \Gam/I(W)$.
Then 
\[
\Jacobi_K(Q,W) \coloneq
\bLam = \Gam/\overline{I(W)}
\] is the \emph{Jacobian algebra} in the sense of 
\cite[Section~4]{DWZ10}. Note that there are examples, where $\bLam$ is finite-dimensional but
$\Lam$ is not.

Let $\alpha,\beta\in Q_1$ with $s(\beta)=t(\alpha)$, then it follows by definition that we have
\begin{equation}\label{eq:partialbetaalpha}
    l_\alpha(\rho_\beta) = \partial_{\beta\alpha}(W),
\end{equation}
where
\begin{equation*}
    \partial_{\beta\alpha}(\alpha_1 \cdots \alpha_m) \coloneq \sum_{\substack{1 \leq p \leq m\\
                  \alpha_{p-1} = \beta \\
                  \alpha_p = \alpha}}
        \alpha_{p+1}\cdots \alpha_ma_1 \cdots \alpha_{p-2}
\end{equation*}
for each cycle $\alpha_1 \cdots \alpha_m$ appearing in $W$.
The index $p$ in the sum above has to be taken modulo $m$.

This equality (\ref{eq:partialbetaalpha}) implies that for a finite-dimensional representation $M$ of a Jacobian algebra $\bLam = \Jacobi_K(Q,W)$ and $k \in Q_0$ the linear map $M(\gamma_k')$ from (\ref{eq:M:gamma:k}) specializes to the map $\gamma_k=(H_k(W))_M$ from \cite[eq. 4.7]{DWZ10}. 

We look now at a slightly generalized version of Jacobian algebras defined by Ladkani \cite{Lad14}.
As before, let $\Gam = K\comp{Q}$, and let $\mathfrak{m}$ be the arrow ideal of $\Gam$.
For $x,y \in \Gam$, let $[x,y] \coloneq xy - xy$.
Following Ladkani \cite{Lad14}, a \emph{hyperpotential} for $Q$ is a collection $\tilde{W} = (\rho_\alpha)_{\alpha \in Q_1}$
of elements in $\Gam = K\comp{Q}$ such that the following hold:
\begin{enumerate}

\item 
For each $\alpha \in Q_1$, we have
$\rho_\alpha \in e_{s(\alpha)} \Gam e_{t(\alpha)}$;

\item 
We have
\[
\sum_{\alpha \in Q_1} [\alpha,\rho_\alpha] = 0.
\]
\end{enumerate}
Such a hyperpotential $\tilde{W}$ is \emph{reduced} if
$\rho_\alpha \in e_{s(\alpha)}\mathfrak{m}^2 e_{t(\alpha)}$ for each
$\alpha \in Q_1$.

The \emph{Jacobian algebra} associated to $Q$ and a reduced hyperpotential $\tilde{W}$ is defined as
\[
\Jacobi_K(Q,\tilde{W}) \coloneq \Gam/\overline{I(\tilde{W})},
\] 
where $I(\tilde{W})$ is the ideal generated be the elements of
\[
R(\tilde{W}) \coloneq \{ \rho_\alpha \mid \alpha \in Q_1 \}.
\]

Note that we only work with reduced potentials and reduced hyperpotentials, because we want that the defining ideal $I(W)$ of $\Gam$ is contained in $\mathfrak{m}^2$. This ensures that the vertices $k$ of $Q$ correspond to the simple modules $S_k$ and the number of arrows $i \to j$ equals 
$\dim \Ext_{\bLam}^1(S_i,S_j)$, where $\bLam = \Jacobi_K(Q,W)$ or 
$\bLam = \Jacobi_K(Q,\tilde{W})$, respectively.

\begin{lemma}[{Ladkani \cite{Lad14}}]
Let $\tilde{W} = (\rho_\alpha)_{\alpha \in Q_1}$ be a reduced hyperpotential for $Q$.
If ${\rm char}(K) = 0$,
then there exists a potential $W$ for $Q$ such that
$\Jacobi_K(Q,\tilde{W}) = \Jacobi_K(Q,W)$.
\end{lemma}

\begin{example}
Assume that $p \coloneq {\rm char}(K) \ge 3$.
Let $\Lam = \bLam = K\comp{Q}/I$, where $Q$ is the quiver
\[
\xymatrix{
1 \ar@(ur,dr)[]^\alpha
}
\]
and $I = \gen{\alpha^{p-1}}$.
Then there is no potential $W$ such that
$\bLam = \Jacobi_K(Q,W)$.
Namely, for $W = \alpha^p$ we have $\partial_\alpha(W) = p\alpha^{p-1} = 0$.
However, for the reduced hyperpotential $\tilde{W} \coloneq (\rho_\alpha)$ with $\rho_\alpha = \alpha^{p-1}$, we get
$\bLam = \Jacobi_K(Q,\tilde{W})$.
\end{example}

\begin{example}
Assume that ${\rm char}(K) = 2$.
Let $\Lam = \bLam = K\comp{Q}/I$, where $Q$ is the quiver
\[
\xymatrix@-3ex{
&& 2\ar[dll]_\alpha
\\
1 \ar@/_2ex/[rrrr]_{\gamma} &&&& 3 \ar[ull]_\beta
}
\]
and $I = \gen{\text{all paths of length $5$ in Q}}$.
Then $\tilde{W} = (\rho_\alpha,\rho_\beta,\rho_\gamma)$
with $\rho_\alpha = \beta\gamma\alpha\beta\gamma$,
 $\rho_\beta = \gamma\alpha\beta\gamma\alpha$,
  $\rho_\gamma = \alpha\beta\gamma\alpha\beta$
is a reduced hyperpotential for $Q$ with $\bLam = \Jacobi_K(Q,\tilde{W})$.  
However, for $W = \alpha\beta\gamma\alpha\beta\gamma$ we get $\partial_\mu(W) = 2 \rho_\mu = 0$ for $\mu \in \{ \alpha,\beta,\gamma\}$ and therefore $\Jacobi_K(Q,W) = K\comp{Q}$.
It follows that $\bLam$ is not a Jacobian algebra in the sense of
Derksen--Weyman--Zelevinsky.
\end{example}

\begin{proposition}\label{prop:Esymm1}
Let $\tilde{W} = (\rho_\alpha)_{\alpha \in Q_1}$ be a reduced hyperpotential for $Q$, and let $\bLam = \Jacobi_K(Q,\tilde{W})$.
Then for $\alpha,\beta \in Q_1$ we have
\[
l_\beta(\rho_\alpha) = r_\alpha(\rho_\beta)
\]
\end{proposition}

\begin{proof}
We need the canonical $K$-linear projection
\begin{align*}
\pi_{\beta,\alpha}\colon \Gam &\to \beta \Gam \alpha
\\
\sum_{p} \lambda_p p &\mapsto \sum_{p = \beta q \alpha} \lambda_p p,
\end{align*}
where the left hand sum runs over all paths $p$ in $Q$ and $\lambda_p \in K$.

By assumption, we have
\begin{equation*}
\sum_{\nu \in Q_1} [\nu,\rho_\nu] = 0 
\end{equation*}
which is equivalent to
\begin{equation*}
\sum_{\nu \in Q_1} \nu\rho_\nu =
\sum_{\nu \in Q_1} \rho_\nu\nu.
\end{equation*}
We obtain
\begin{align*}
\beta l_\beta(\rho_\alpha)\alpha = \pi_{\beta,\alpha}(\rho_\alpha\alpha) 
&= \pi_{\beta,\alpha}\left(\sum_{\nu \in Q_1} \rho_\nu\nu\right)
\\
&= \pi_{\beta,\alpha}\left(\sum_{\nu \in Q_1} \nu\rho_\nu\right)
= \pi_{\beta,\alpha}(\beta\rho_\beta) = \beta r_\alpha(\rho_\beta)\alpha.
\end{align*}
Since $\Gam$ has no zero-divisors, the equality 
$\beta l_\beta(\rho_\alpha)\alpha = \beta r_\alpha(\rho_\beta)\alpha$
implies that $l_\beta(\rho_\alpha) = r_\alpha(\rho_\beta)$.
\end{proof}

Suppose that $0 \not= \rho \coloneq l_\beta(\rho_\alpha) = r_\alpha(\rho_\beta)$.
Then the situation
can be vizualized as 
\[
\xymatrix@-3ex{
&& k\ar[dll]_\alpha\\
i \ar@{-->}@/_3ex/[rrrr]_{\rho} &&&& j \ar[ull]_\beta
}
\]
where $\rho_\alpha = \beta \rho$ and $\rho_\beta = \rho \alpha$.

\begin{corollary}\label{cor:Esymm2}
Let $\tilde{W} = (\rho_\alpha)_{\alpha \in Q_1}$ be a reduced hyperpotential for $Q$, and let 
$\Lam = \Gam/I(\tilde{W})$.
Then for each $M \in {\rm mod}(\Lam)$ and $k \in Q_0$ we have
\[
M(\gamma_k) \coloneq M(\gamma_k') = M(\gamma_k'').
\]
In particular, we get a complex
\[
0 \to M(k) \xrightarrow{M(\beta_k)} M_{\rm out}(k) \xrightarrow{M(\gamma_k)} 
M_{\rm in}(k) \xrightarrow{M(\alpha_k)} M(k) \to 0.
\]
\end{corollary}

\begin{proof}
This follows directly from Proposition~\ref{prop:Esymm1} and the definitions of
the maps $M(\gamma_k')$ and $M(\gamma_k'')$.
\end{proof}

Theorems~\ref{thm:main3} and \ref{thm:main4}
specialize to the following version of \cite[Proposition 10.4]{DWZ10} and its dual \cite[Remark~10.8]{DWZ10} with weaker hypotheses.
More precisely, we do 
not need Derksen--Weyman--Zelevinsky's assumptions that the quiver $Q$ is $2$-acyclic and 
that $\characteristic(K) = 0$.

\begin{theorem}\label{thm:main5}
    Let $Q$ be a quiver, and let $\tilde{W}$ be a reduced hyperpotential for $Q$. Let $M$ be a finite-dimensional representation of the Jacobian algebra $\bLam = \Jacobi_K(Q,\tilde{W})$. 
    Then $M$ has a minimal projective presentation
    \begin{equation*}
        \bigoplus_{k \in Q_0} (\bLam e_k)^{\dim(\ker(M(\alpha_k))/\im(M(\gamma_k')))} \rightarrow \bigoplus_{k \in Q_0} (\bLam e_k)^{\dim\coker(M(\alpha_k))} \rightarrow M \rightarrow 0,
    \end{equation*}
    in ${\rm PC}(\bLam)$
    and a minimal injective presentation
\begin{equation*}
0\rightarrow M\rightarrow \bigoplus_{k\in Q_0} D_2(e_k \bLam)^{\dim \ker(M(\beta_k))}
\rightarrow  \bigoplus_{k\in Q_0} D_2(e_k \bLam)^{\dim(\ker(M(\gamma_k''))/\im(M(\beta_k)))}
\end{equation*}
in ${\rm Dis}(\bLam)$,
where the map $M(\gamma_k) \coloneq M(\gamma_k') = M(\gamma_k'')$ is defined in terms of the set of relations $R = R(\tilde{W})$.
\end{theorem}

\subsection{$E$-symmetry for Jacobian algebras}
For $a = (a_k)_k$ and $b = (b_k)_k$ in $\Z^{Q_0}$ let
\[
\gen{a,b} \coloneq \sum_{k \in Q_0} a_kb_k.
\]

Let $\Lam = \Gam/I(\tilde{W})$ and $\bLam = \Jacobi_K(Q,\tilde{W})$ be a Jacobian algebra as above.
For $M \in {\rm mod}(\Lam)$ let 
$\Dim(M) \coloneq ([M:S_k])_{k \in Q_0}$ be the
\emph{dimension vector} of $M$, where $[M:S_k]$ denotes the Jordan-H\"older multiplicity of the
simple module $S_k$ for $M$.
The \emph{$E$-invariants} of $M$ are defined as
\begin{align*}
E(M) &\coloneq \dim \End_{\Lam}(M) + \gen{\g^{\rm proj}(M),\Dim(M)},
\\
E^-(M) &\coloneq \dim \End_{\Lam}(M) + \gen{\g^{\rm inj}(M),\Dim(M)}.
\end{align*}

These $E$-invariants play a major role in the categorification of Fomin-Zelevinsky cluster algebras.
It was already observed in \cite{DWZ10} that for certain finite-dimensional Jacobian algebras
$B = \bLam = \Jacobi_K(Q,W)$ (the condition is that $I(W)$ is an admissible ideal in $KQ$), we have
\begin{align*}
E(M) &= \dim \hom_B(M,\tau_B(M)),
\\
E^-(M) &= \dim \hom_B(\tau_B^-(M),M)
\end{align*}
for all $M \in {\rm mod}(B)$, where $\tau_B$ denote the Auslander-Reiten translation for $B$.

\begin{theorem}[$E$-symmetry Theorem]\label{thm:main6}
Let $\bLam = \Jacobi_K(Q,\tilde{W})$ 
for some reduced hyperpotential $\tilde{W}$ for $Q$.
Then for each $M \in {\rm mod}(\bLam)$ we have
\[
E(M) = E^-(M).
\]
\end{theorem}

\begin{proof}
By Corollary~\ref{cor:Esymm2} we know that $M(\gamma_k) \coloneq M(\gamma_k') = M(\gamma_k'')$,
where the maps $M(\gamma_k')$ and $M(\gamma_k'')$ are defined in terms of the set of relations $R = R(\tilde{W})$.
Thus we get the complex
\[
0 \to M(k) \xrightarrow{M(\beta_k)} M_{\rm out}(k) \xrightarrow{M(\gamma_k)} 
M_{\rm in}(k) \xrightarrow{M(\alpha_k)} M(k) \to 0.
\]

We use Theorem~\ref{thm:main5}
and calculate
\begin{align*}
\gen{\g^{\rm proj}(M),\Dim(M)} &= \sum_{k \in Q_0} \g_k^{\rm proj}(M) \cdot \dim M(k)
\\
&=
\sum_{k \in Q_0}\left(\dim \ker(M(\alpha_k)) - \dim \im(M(\gamma_k)) - \dim \coker(M(\alpha_k))\right) \cdot \dim M(k)
\\
&= \sum_{k \in Q_0} \left(\dim M_{\rm in}(k) - \dim \im(M(\gamma_k)) - \dim M(k)
\right) \cdot \dim M(k)
\end{align*}
and
\begin{align*}
\gen{\g^{\rm inj}(M),\Dim(M)} &= \sum_{k \in Q_0} \g_k^{\rm inj}(M) \cdot \dim M(k)
\\
&=
\sum_{k \in Q_0}\left(\dim \ker(M(\gamma_k)) - \dim \im(M(\beta_k)) - \dim \ker(M(\beta_k))\right) \cdot \dim M(k)
\\
&= \sum_{k \in Q_0} \left(\dim M_{\rm out}(k) - \dim \im(M(\gamma_k)) - \dim M(k)
\right) \cdot \dim M(k).
\end{align*}
(Here we used the usual dimension formulas from Linear Algebra.)
Thus 
\begin{equation}\label{eq:gvectorbilinear}
\gen{\g^{\rm proj}(M),\dim(M)} = \gen{\g^{\rm inj}(M),\dim(M)}
\end{equation} 
if and only if
\[
\sum_{k \in Q_0} \dim M_{\rm in}(k) \cdot \dim M(k) = \sum_{k \in Q_0} \dim M_{\rm out}(k) \cdot \dim M(k).
\]
Since both of these sums are obviously equal to
\[
\sum_{\alpha \in Q_1} \dim M(s(\alpha)) \cdot \dim M(t(\alpha)),
\]
the equality (\ref{eq:gvectorbilinear}) holds.
We conclude that $E(M) = E^-(M)$.
\end{proof}

\section{Examples}

Recall that for
$\Lam = \Gam/\mathfrak{J}$ and $M \in {\rm mod}(\Lam)$ we have for each $k \in Q_0$ a diagram
\begin{equation*}
\xymatrix@+2ex{
M_{\rm rel}'(k)\ar[r]^{M(\gamma_k')} & M_{\rm in}(k)\ar[r]^{M(\alpha_k)}
& M(k)\ar[r]^<<<<<{M(\beta_k)}
& M_{\rm out}(k)\ar[r]^<<<<<{M(\gamma_k'')}
& M_{\rm rel}''(k)
}.
\end{equation*}
For computing a minimal projective presentation (resp. injective presentation) of $M$ we just need
the left (resp. right) half of the diagram.
Recall also that
\begin{align*}
    h_k^{(0)}(M) &\coloneq \dim\ker(M(\beta_k)), \\
    h_k^{(1)}(M) &\coloneq \dim\ker(M(\gamma''_k))/\im(M(\beta_k)),\\ 
    h_k^{(2)}(M) &\coloneq \dim\ker(M(\alpha_k))/\im(M(\gamma'_k)),\\
    h_k^{(3)}(M) &\coloneq \dim\coker(M(\alpha_k)).
\end{align*}
Note that
for Jacobian algebras we have $M_{\rm rel}'(k) = M_{\rm out}(k)$ and $M_{\rm rel}''(k) = M_{\rm in}(k)$.



\begin{example}\label{ex:6.1}
Let $\Lambda=KQ/I$, where $Q$ is the quiver 
    \begin{equation*}
        \begin{tikzcd}[sep=scriptsize]
	        1 & 2 
	        \arrow["\alpha", bend left, from=1-1, to=1-2]
	        \arrow["\beta", bend left, from=1-2, to=1-1]
    \end{tikzcd}
    \end{equation*} and $I = \gen{\alpha\beta\alpha,\beta\alpha\beta}$. 
Note that for the potential $W \coloneq (\alpha\beta)^2$ we have $I = I(W)$ provided ${\rm char}(K) \not= 2$. 
We want to construct a minimal projective presentation of the $\Lam$-module
    \begin{equation*}
        M= \begin{matrix}
            1 \\ 2
        \end{matrix} \; .
    \end{equation*}
We have    
    $M(1) = M_{\text{in}}(2) = K e_1$, 
$M(2) = M_{\text{in}}(1) = K\alpha $, $M_{\rm rel}'(1) = M(2)$, $M_{\rm rel}'(2) = M(1)$ with maps
    \begin{alignat*}{3}
    M(\alpha_1) &=M(\beta)\colon &M(2) &\rightarrow M(1),\quad &\alpha &\mapsto 0, \\
    M(\alpha_2) &= M(\alpha)\colon &M(1)    &\rightarrow M(2), &e_1 &\mapsto \alpha, \\   
    M(\gamma_1')&= M(\alpha\beta) \colon\; &M(2) & \rightarrow M(2), & \alpha &\mapsto 0,\\
    M(\gamma_2')&=  M(\beta\alpha)\colon &M(1) & \rightarrow M(1),&   e_1 &\mapsto  0.
    \end{alignat*}
This yields $\rk(M(\alpha_1)) = \rk(M(\alpha_2)) = 1$, $\rk(M(\gamma_1')) = \rk(M(\gamma_2')) = 0$. 
Thus, we have
\begin{align*}
h_1^{(2)}(M) &= 1-0=1, & h_1^{(3)}(M) &=1-0=1, 
\\
h_2^{(2)}(M) &=0-0=0, & h_2^{(3)}(M) &=1-1=0.
\end{align*}
We obtain a minimal projective presentation of the form
\begin{equation*}
        P_1 \rightarrow P_1 \rightarrow M \rightarrow 0.
    \end{equation*}
More precisely (in the notation of Section~\ref{subsec:5.2}), we have    
    $\ker(M(\alpha_1)) = K\alpha$ and $\im(M(\gamma_1'))=  0$, so $M_{\text{in}}^{(1)}(1)=K\alpha$ and $M^{(0)}(1) = Ke_1$. Similarly, $\ker(M(\alpha_2)) = 0 = \im(M(\gamma_2'))$ and $\im(M(\alpha_2)) = K\alpha$, so $M^{(0)}(2)=0$.

First we obtain now the standard presentation
\[
P_M \xrightarrow{d_M} K_M \xrightarrow{e_M} M \rightarrow 0
\]
with 
\[
P_M= (P_1\otimes\alpha\otimes K e_1) \oplus (P_2\otimes\beta\otimes K \alpha), \quad\text{ and }\quad
K_M = (P_1\otimes K e_1) \oplus (P_2\otimes K \alpha)
\]
and 
\begin{align*}
d_M(e_1\otimes\alpha\otimes e_2) &= \alpha\otimes e_1 - e_2\otimes \alpha, &
e_M(e_1\otimes e_1) &= e_1,\\
d_M(e_2\otimes\beta\otimes\alpha) &=\qquad\qquad\;\beta\otimes\alpha,
 & e_M(e_2\otimes\alpha) &= \alpha.
\end{align*}    
Note, that the $\Lam$-linear maps $d_M$ and $e_M$ are determined by their values on the generators of $P_M$ resp.\@ of $K_M$. In this situation we may choose the summands of $P_M$ resp.\@ $K_M$ as follows:
\begin{align*}
P_M^{(0)} &= P_2\otimes\alpha\otimes Ke_1 & K_M^{(0)} &= P_1\otimes K e_1,\\
P_M^{(1)} &= P_1\otimes\beta\otimes K\alpha & K_M^{(0)} &= P_2\otimes K \alpha,\\
P_M^{(2)} &= 0.
\end{align*}
Recall, that $K'_M:=d_M(P_M^{(0)})$ and we know that $K_M= K_M^{(0)}\oplus K'_M$. 
We obtain by Theorem~\ref{thm:main2} a minimal projective presentation of $M$ by considering
the composition
\[
P_M^{(1)} \xrightarrow{d_M\mid_1} K_M \xrightarrow{\mathrm{pr}} K_M^{(0)}, 
\]
where $\mathrm{pr}$ is the projection along $K'_M$. Now, it is sufficient to observe that
\[
d_M(e_1\otimes\beta\otimes\alpha)+K'_M= \beta\otimes\alpha + K'_M= \beta\alpha\otimes e_1+K'_M
\]
since $\beta\alpha\otimes e_1-\beta\otimes\alpha=d_M(\beta\otimes\alpha\otimes e_1)\in K'_M$.
Thus, we get indeed a minimal projective presentation 
    \[
    P_1 \otimes \beta \otimes K\alpha \xrightarrow{\overline{d}_M}  P_1 \otimes Ke_1 \xrightarrow{e_M^{(0)}} 
    M \to 0
    \]
with 
\[
\overline{d}_M(p \otimes \beta \otimes \alpha) = 
\begin{cases}
\beta\alpha \otimes e_1 & \text{if $p= e_1$},
\\
0 & \text{if $p = \alpha$},
\\
0 &  \text{if $p = \beta\alpha$},
\end{cases}
\]
and
\[
e_M^{(0)}(p \otimes e_1) = M(p)e_1 = 
\begin{cases}
e_1 & \text{if $p= e_1$},
\\
\alpha & \text{if $p = \alpha$},
\\
0 &  \text{if $p = \beta\alpha$}.
\end{cases}
\]
\end{example}

\begin{example}\label{ex:6.2}
Let $\Lam = KQ/I$ be the same algebra as in Examples~\ref{ex:6.1}.
Now we consider 
\[ 
M=P_1 = \begin{matrix} 1\\2\\1\end{matrix}.
\]
We have    
    $M(1) = M_{\text{in}}(2) = K\gen{e_1,\beta\alpha}$, 
$M(2) = M_{\text{in}}(1) = K\alpha$, $M_{\rm rel}'(1) = M(2)$, $M_{\rm rel}'(2) = M(1)$ with maps
\begin{alignat*}{4}
M(\alpha_1)&=M(\beta)\colon   & M(2) &\rightarrow M(1), \quad & \alpha &\mapsto \beta\alpha, \\
M(\alpha_2)&=M(\alpha)\colon  & M(1) &\rightarrow M(2), &e_1 &\mapsto \alpha,\quad &\beta\alpha &\mapsto 0,\\   
M(\gamma_1')&= M(\alpha\beta) \colon\; &M(2) &\rightarrow M(2), &  \alpha&\mapsto 0, \\
M(\gamma_2')&= M(\beta\alpha) \colon &M(1) &\rightarrow M(1), &    e_1 &\mapsto \beta\alpha,  &\beta\alpha &\mapsto 0.
    \end{alignat*}
This yields $\rk(M(\alpha_1)) = \rk(M(\alpha_2)) = 1$, $\rk(M(\gamma_1')) = 0$, $\rk(M(\gamma_2')) = 1$.     
It follows that
\begin{align*}
h_1^{(2)}(M) &= 0, & h_1^{(3)}(M) &=1, 
\\
h_2^{(2)}(M) &=0, & h_2^{(3)}(M) &=0.
\end{align*}
Thus the minimal projective presentation of $M$ is of the form
    \begin{equation*}
        0 \rightarrow P_1 \rightarrow M \rightarrow 0.
    \end{equation*}
\end{example}

\begin{example}
Let $\Lam = K\comp{Q}/\fJ$ where $Q$ is given by 
        \begin{equation*}
            \begin{tikzcd}[sep=scriptsize]
                1 \ar["\beta", loop right, myloop] \ar["\alpha" ', loop left, myloop2] 
            \end{tikzcd}        
        \end{equation*} 
and $\fJ = \gen{\alpha\beta,\beta\alpha}$. Note that $\fJ$ is a closed ideal, i.e. $\bLam = \Lam$,
so that $\Lam$ is a local (commutative)  basic pseudocompact algebra. 
In fact,
it is isomorphic to $K[\![x,y]\!]/(xy)$. 
Note that $P_1 = \Lam$.
We consider the string
module $M$, which is described  by the following diagram:
\[M\coloneq\quad\begin{tikzcd}[cramped,sep=scriptsize]
	&&& 4 && 6 \\
	1 && 3 && 5 \\
	& 2
	\arrow["\alpha"', from=1-4, to=2-3]
	\arrow["\beta"', from=1-4, to=2-5]
	\arrow["\alpha"', from=1-6, to=2-5]
	\arrow["\beta"', from=2-1, to=3-2]
	\arrow["\alpha"', from=2-3, to=3-2]
\end{tikzcd}\]
Thus, the structure maps of $M$ are
\[
M(\alpha)=
\left(\begin{matrix} 
0 &0 &0 &0 &0 &0\\
  &0 &1 &0 &0 &0\\
  &  &0 &1 &0 &0\\
  &  &  &0 &0 &0\\
  &0 &  &  &0 &1\\
  &  &  &  &  &0
\end{matrix}\right)\qquad\text{and}\qquad
M(\beta)=
\left(\begin{matrix} 
0 &  &  &  &  & \\
1 &0 &  &  &0 & \\
0 &0 &0 &  &  & \\
0 &0 &0 &0 &  & \\
0 &0 &0 &1 &0 &\\
0 &0 &0 &0 &0 &0
\end{matrix}\right).
\]
For the unique vertex $k = 1$ of our quiver $Q$ we get
\[
M(\alpha_k)=(M(\alpha), M(\beta)),\quad
M(\beta_k)=\begin{pmatrix}
M(\alpha)\\ M(\beta)    
\end{pmatrix},
\]
\[
M(\gamma_k')=
\begin{pmatrix}
    M(l_\alpha(\alpha\beta)) & M(l_\alpha(\beta\alpha))
    \\ M(l_\beta(\alpha\beta)) &M(l_\beta(\beta\alpha)) 
\end{pmatrix}
=
\begin{pmatrix}
    M(\beta) & 0\\ 0 &M(\alpha) 
\end{pmatrix},
\]
\[
M(\gamma_k'')=
\begin{pmatrix}
    M(r_\alpha(\beta\alpha)) & M(r_\beta(\beta\alpha))
    \\ M(r_\alpha(\alpha\beta)) &M(r_\beta(\alpha\beta)) 
\end{pmatrix}
=
\begin{pmatrix}
    M(\beta) & 0\\ 0 &M(\alpha) 
\end{pmatrix}.
\]
We obtain
\[
\rk(M(\alpha_k))=3,\quad
\rk(M(\beta_k))=4,\quad
\rk(M(\gamma'_k)) = \rk(M(\gamma''_k)) = 5.
\]
With this at hand, we find
\begin{align*}
h_k^{(2)}(M) &=(12-3)-5=4, & h_k^{(3)}(M) &=6-3=3.
\end{align*}
This yields a minimal projective presentation of the form
\[
\Lam^4\rightarrow\Lam^3\rightarrow M\rightarrow 0.
\]
Furthermore, we have
\begin{align*}
h_k^{(0)}(M) &= 6-4=2, & h_k^{(1)}(M) &=(12-5)-4=3,
\end{align*}
This yields a minimal
injective presentation of the form
\[
0\rightarrow M\rightarrow D_2(\Lam)^2\rightarrow D_2(\Lam)^3.
\]

Since $\mathfrak{m}^3 M=0$ we recover for any $l\geq 4$ over the truncation $\Lam_l$ the well-known minimal projective presentation
resp. minimal injective presentation of the string module $M$ over
the special biserial algebra $\Lam_l$ of the same type. 

The discrete-injective module $I_{1,{\rm dis}} = D_2(\Lam)$ has a countable 
$K$-basis $\ldots, v_2, v_1, v_0=w_0, w_1, w_2,\ldots$ with 
\begin{align*}
\alpha\cdot v_k &=v_{k-1}, &\alpha\cdot w_k &=0,\\
\beta\cdot v_k  &=0,       &\beta\cdot w_k &=w_{k-1}
\end{align*}
for all $k \geq 1$ and $\alpha\cdot v_0=\beta\cdot v_0 = 0$.
Since $\Lam$ is left noetherian, one can show that 
$I_{1,{\rm dis}}$ is the injective envelope $I_{1,\Lam}$ of $S_1$ in
$\Mod(\Lam)$.
\end{example}

\bigskip

\subsection*{Statement on AI usage}
Example~\ref{ex:3.4} is based on a similar example (for the $2$-loop quiver), which we found via the AI model ChatGPT Pro.
This is the only place in the paper where AI was used.

\subsection*{Acknowledgements}
JS was partially funded by the Deutsche Forschungsgemeinschaft
(DFG, German Research Foundation) under Germany's Excellence Strategy - GZ 2047/1, Projekt-ID 390685813. CH was funded by the Deutsche Forschungsgemeinschaft (DFG, German Research Foundation) under
Germany's Excellence Strategy - GZ 2047/1, Projekt-ID 390685813. DLF was partially funded by Fondazione CARIPARO through a “Starting Package per attrarre ricercatrici e ricercatori eccellenti dall’estero” (CUP:C93C22008360007) granted to him at Università degli Studi di Padova.

\end{document}